\documentclass[12pt,a4paper,english,reqno]{amsart}
\usepackage{amsthm}
\usepackage{amssymb}
\usepackage{amscd}
\usepackage{amstext}
\usepackage[all]{xy}
\xyoption{2cell}
\UseTwocells

\usepackage{graphicx}
\usepackage{color}
\theoremstyle{plain}
\newtheorem{thm}{Theorem}[section]
\newtheorem{prp}[thm]{Proposition}
\newtheorem{lem}[thm]{Lemma}
\newtheorem{cor}[thm]{Corollary}

\newtheorem*{thm-nn}{Theorem}
\newtheorem*{prp-nn}{Proposition}
\newtheorem*{lem-nn}{Lemma}
\newtheorem*{cor-nn}{Corollary}
\newtheorem*{clm-nn}{Claim}
\newtheorem*{cnj-nn}{Conjecture}
\newtheorem*{prb-nn}{Problem}

\theoremstyle{definition}
\newtheorem{dfn}[thm]{Definition}
\newtheorem{exm}[thm]{Example}

\newtheorem*{dfn-nn}{Definition}

\newtheorem{rmk}[thm]{Remark}

\newcommand{\xyR}[1]{%
\xydef@\xymatrixrowsep@{#1}}

\newcommand{\xyC}[1]{%
\xydef@\xymatrixcolsep@{#1}}

\def\al{\alpha}
\def\be{\beta}

\def\la{\lambda}
\def\ro{\rho}
\def\si{\sigma}

\def\ta{\tau}

\def\ph{\phi}
\def\ps{\psi}
\def\Ga{\Gamma}
\def\De{\Delta}

\def\Aut{\operatorname{Aut}}

\def\Rep{\operatorname{Rep}}

\def\Mod{\operatorname{Mod}}

\def\calC{{\mathcal C}}

\def\calW{{\mathcal W}}

\def\bbN{{\mathbb N}}
\def\bbZ{{\mathbb Z}}
\def\bbP{{\mathbb P}}

\def\op{^{\mathrm{op}}} 

\def\inv{^{-1}}

\def\implies{\text{$\Rightarrow$}}
\def\impliedby{\text{$\Leftarrow$}}

\def\incl{\hookrightarrow}
\def\iso{\cong}

\def\ovl{\overline}
\def\Ds{\bigoplus}

\def\dsm#1,#2..#3{\bigoplus_{{#1}={#2}}^{#3}}
\def\sm#1,#2..#3{\sum_{{#1}={#2}}^{#3}}

\def\id{1\kern-.25em{\text{{\rm l}}}} 

\def\isoto{\ \raise.3ex\hbox{$^{\sim}$}\kern-.8em\hbox{$\to$}\ } 

\def\ang#1{{\langle #1 \rangle}}
\def\ya#1{\xrightarrow{#1}}
\def\blank{\operatorname{-}}

\def\bg{%
\family{cmr}\size{20}{12pt}\selectfont}

\def\bigzerou{%
\smash{\lower1.7ex\hbox{\bg 0}}}

\def\repr[#1;#2;#3;#4;#5]{
\left(
\begin{matrix}#1\\#2\end{matrix}
#3
\begin{matrix}#4\\#5\end{matrix}
\right)}
\def\smat#1{\begin{smallmatrix} #1 \end{smallmatrix}}

\def\k{\Bbbk}

\keywords{coverings, gradings, smash products, quiver presentations, Brauer graphs}

\begin{document}
\title{Smash products of group weighted bound quivers and Brauer graphs}
\author{Hideto Asashiba}

\begin{abstract}
Let  $\k$ be a field, $G$ a group, and
$(Q, I)$ a bound quiver.
A map $W\colon Q_1 \to G$ is called a $G$-{\em weight} on $Q$,
which defines a $G$-graded $\k$-category $\k(Q, W)$,
and $W$ is called {\em homogeneous} if $I$ is a homogeneous ideal
of the $G$-graded $\k$-category $\k(Q, W)$.
Then we have a $G$-graded
$\k$-category $\k(Q, I, W):= \k(Q, W)/I$.
We can then form a smash product
$\k(Q, I, W)\# G$ of $\k(Q, I, W)$ and $G$, which canonically defines a Galois
covering $\k(Q, I, W)\# G \to \k(Q, I)$ with group $G$
(we will see that all such Galois coverings to $\k(Q, I)$ have this form for some $W$).
First we give a quiver presentation 
$
\k(Q_{G, W}, I_{G, W}) \iso \k(Q, I, W)\# G
$
of the smash product $\k(Q, I, W)\# G$.
Next if $(Q, I, W)$ is defined by
a Brauer graph with an admissible weight,
then the smash product $\k(Q, I, W)\# G$ is again defined by
a Brauer graph,
which will be computed explicitly.
The computation is simplified by introducing a concept of Brauer permutations
as an intermediate one between Brauer graphs and Brauer bound quivers.
This extends and simplifies the result by Green--Schroll--Snashall on the
computation of coverings of Brauer graphs, which dealt with
the case that $G$ is a finite abelian group, while in our case $G$ is an arbitrary group.
In particular, it enables us to delete all cycles in Brauer graphs to transform it to an infinite Brauer tree.
\end{abstract}

\makeatletter
\@namedef{subjclassname@2010}{%
  \textup{2010} Mathematics Subject Classification}
\makeatother
\subjclass[2010]{18D05, 16W22, 16W50}

\thanks{This work is partially supported by Grant-in-Aid for Scientific Research 25610003 and 25287001 from JSPS}

\email{asashiba.hideto@shizuoka.ac.jp}
\address{Department of Mathematics, Graduate School of Science, Shizuoka University,
836 Ohya, Suruga-ku, Shizuoka, 422-8529, Japan.}
\maketitle

\section*{Introduction}
We fix a field $\k$ and a group $G$.
To include infinite coverings of $\k$-algebras into consideration we usually regard $\k$-algebras as locally bounded $\k$-categories with finite objects (see \cite{Ga81} for definitions).
The set of positive integers is denoted by $\bbN$.

\subsection*{Coverings, gradings and smash products}
Covering theory was introduced into representation theory of algebras
by papers Gabriel--Riedtmann \cite{GR79}, Riedtmann \cite{Rie80},
Bongartz-Gabriel \cite{BG81} and Gabriel \cite{Ga81}.
Since then it became an important tool to reduce many problems of algebras to
the corresponding ones for algebras/categories with simpler structures, e.g., 
for bound quiver categories whose quivers do not have oriented cycles.
Let $A$ be a locally bounded $\k$-category with a free $G$-action.
Then the orbit category $A/G$ and the canonical functor $F\colon A \to A/G$
is defined (see \cite[3.1]{Ga81}).  A functor $E\colon A \to B$ is called a {\em Galois covering} functor with group $G$
if it is isomorphic to $F$, namely if
there exists an isomorphism $H\colon A/G \to B$ such that $E = HF$.
In application it becomes a problem how to construct $A$ from $B$ in the setting above
when $B$ is given by a bound quiver $(Q, I)$.
The construction may depend on the presentation of $B$.
There are some constructions.

First of all one is given by a topological construction, or its combinatorial form
as stated and used in Waschb\"{u}sch \cite{Wa81}, Green \cite{Gr83},
Martinez-Villa--de la Pe\~{n}a \cite{MdP83} (cf.\ Bretscher--Gabriel \cite[3.3(a)]{BrG83}),
namely, first construct a covering $F\colon (Q', I') \to (Q, I)$ of bound quivers
by the following steps:
(1) to form a universal covering $\tilde{Q}$ with a canonical quiver morphism
$\tilde{F}\colon \tilde{Q} \to Q$ using a vertex $x_0$ of $Q$ as
a base point in a topological sense;
and (2) to make an orbit quiver $Q':= \tilde{Q}/N$ of $\tilde{Q}$ with a quiver morphism
$F\colon Q' \to Q$ induced from $\tilde{F}$
by using a suitable normal subgroup $N$ of the fundamental group
$\pi_1(Q, x_0)$ containing the normal subgroup $N(Q, I, x_0)$ defined by so-called {\em minimal relations} in $I$
(see Definition \ref{dfn:min-rel-hmg}(1)) that has a free action of the group $G:= \pi_1(Q, x_0)/N$
; and finally (3) to generate an ideal $I'$ of the path category $\k Q'$
by the morphisms of the form $L(\ro)$ for all liftings $L$ (see Definition \ref{dfn:covering}(4), (5)) of $F$
and minimal relations $\ro$ in $I$ (note here that zero relations are minimal relations in our definition).
Then (4) the functor $\k(Q', I') \to \k(Q, I) = B$
induced from the $\k$-linearization $\k F\colon \k Q' \to \k Q$ of $F$
is a Galois covering with group $G$, in particular, $\k (Q', I')/G \iso \k (Q, I)$.

Another more direct construction of Galois coverings using group gradings
is also given in \cite{Gr83}.
In this paper Green defined a notion of coverings of bound quivers
and has shown the existence of the universal covering and that
all coverings from connected bound quivers are Galois coverings
(Note that in the definitions of coverings of bound quivers it is assumed that
they are regular coverings between quivers.  See Definition \ref{dfn:covering}.)
He also suggests us an existence of a bijection between the set
$$
\calC:=\{(F, L) \mid F\colon (Q', I') \to (Q, I) \text{ a covering with $Q'$ connected}, L \text{ a lifting of }F\}
$$
and the set
$$
\calW:= \{(G, W) \mid G \text{ a group}, W \text{ a homogeneous $G$-weight such that } Q_{G, W}
\text{ connected}\}
$$
up to suitable equivalence relations on them, where
a $G$-{\em weight} on $Q$ is a map $Q_1 \to G$, and is said to be {\em homogeneous}
if $I$ is a homogeneous ideal of the path category $\k Q$
with the $G$-grading naturally defined by $W$.
Denote these correspondences by
$$
(F, L) \overset{\ph}{\mapsto} (\Aut(F), W_{F, L}), \text{ and }
(G, W) \overset{\ps}{\mapsto} (F_{G,W} \colon (Q_{G,W}, I_{G,W}) \to (Q, I), L_{G,W}),
$$
where
$\Aut(F):= \{p \in \Aut(Q', I')\mid Fp = F\}$.
(These are given in the proof of \cite[Theorem 3.2, Theorem 3.4]{Gr83}, respectively.
See Definition \ref{dfn:cov-by-weight} for details of the latter.)
Then Theorem 3.4 states that $\ph\ps = \id$.
But $\ps\ph = \id$ is not directly stated.
Instead, in Theorem 3.2, it was shown that
the category $\Rep_G (Q, I)$ of $G$-graded representations of $(Q, I)$
is equivalent to the category $\Rep (Q', I')$ of representations of $(Q', I')$
(although he dealt only with the point-wise finite-dimensional representations).
This means that the category $\Mod_G \k(Q, I)$ of $G$-graded $\k(Q, I)$-modules
is equivalent to the category $\Mod \k(Q', I')$ of $\k(Q', I')$-modules.
By the equality $\ph\ps = \id$ we can replace $(Q', I')$ by $(Q_{G,W}, I_{G,W})$
starting from $(G, W)$ such that $Q_{G,W}$ is connected.
Then we have equivalences of categories:
\begin{equation}\label{eq:cat-eqs}
\Mod \k(Q', I') \simeq \Mod_G \k(Q, I) \simeq \Mod \k(Q_{G,W}, I_{G,W}),
\end{equation}
and hence an isomorphism $\k(Q', I') \iso \k(Q_{G,W}, I_{G,W})$ because these are skeletal,
which means that $\ps\ph = \id$ up to a suitable equivalence relation on $\calC$.
In particular, from this we know that the class of coverings $F\colon (Q', I') \to (Q, I)$
with $Q'$ connected (namely of Galois coverings constructed by the first one above)
is exactly the class of coverings of the form
$F_{G,W} \colon (Q_{G,W}, I_{G,W}) \to (Q, I)$ for some $(G, W)$ with
$Q_{G,W}$ connected,
which gives us the second construction of Galois coverings that is more direct
than the first one.

Another theoretical construction due to Cibils--Marcos \cite{CM06}
uses the smash product of a $G$-graded category and $G$.
Note that if $A$ is a $\k$-category with a $G$-action, then the orbit category $B \iso A/G$ has
a natural $G$-grading.
If $B$ is a $G$-graded category, then the smash product $B\# G$ of $B$ and $G$ is defined,
which has a free $G$-action, and the canonical functor $B \# G \to B$ turns out to be
a Galois covering with group $G$, thus in particular, $(B \# G)/G \iso B$.
Therefore the third construction is given as follows:
(1) to give a $G$-grading on $B$; (2) to form a smash product $B \# G$.
Then (3) the canonical functor $B\# G \to B$ is a Galois covering with group $G$.

We can combine the second and the third constructions as follows, which is one of the
purposes of this paper.
The connections of module categories above were generalized
by Civils--Marcos \cite{CM06} by using
smash products of $G$-graded categories and the group $G$,
which was further generalized in \cite{Asa11, Asa15} as a 2-categorical version
of Cohen--Montgomery duality \cite{CoM84} to show that orbit category constructions and
smash product constructions are mutually inverse (see also Tamaki \cite{Tam09}).
In particular, in \cite{CM06} (or in \cite{Asa11}) it was shown
that the category $\Mod_G \k(Q, I)$ is equivalent to the category $\Mod \k(Q, I)\# G$
of modules over the smash product $\k(Q, I)\# G$.
Then by equivalences \eqref{eq:cat-eqs}
we see that $\Mod \k(Q_{G,W}, I_{G,W}) \simeq \Mod \k(Q, I)\# G$,
which indirectly shows that $(Q_{G,W}, I_{G,W})$ is a quiver presentation of $\k(Q, I)\# G$
when $Q_{G,W}$ is connected.
In Section 1 we will show this fact with a direct proof without
the connectedness assumption on $Q_{G,W}$
(Theorem \ref{thm:quiv-pres-sm-prod}).
It seems there are no explicit quiver presentations
of smash products in literature so far,
while quiver presentations for orbit categories (or more generally Grothendieck constructions)
are computed such as in Reiten--Riedtmann \cite{RR85} or in our paper \cite{Asa-Kim}.
Theorem \ref{thm:quiv-pres-sm-prod} gives us quiver presentations of smash products
of $G$-graded locally bounded category with $G$-gradings defined by $G$-weights.
As a consequnce, we see that the Galois coverings constructed by the first one
are exactly those constructed by the third one with gradings given by $G$-weights.
Note that there exist other types of $G$-gradings that are not defined by $G$-weights
(see e.g., Dugas \cite[Section 6]{Dug10} for an important example that reduces even
a nonstandard representation-finite self-injective algebra to a Brauer tree algebra).
Therefore coverings given by smash products are wider than those given by topological way.

\subsection*{Brauer graph algebras}
A Brauer graph is essentially a non-oriented graph with two maps from the set of vertices $V$:
the first (resp.\ second) one assigns to each vertex $x$ a cyclic permutation of ``half edges'' connected to $x$
(resp.\ a natural number), the second one is called the {\em multiplicity} of the Brauer graph.
To deal with Brauer graphs without ambiguity we have to distinguish two ends of loops.
To this end the notion of half edges is introduced (e.g., see Adachi--Aihara--Chan \cite{AAC}), thus
the formulation using half edges is necessary only when the graph in question has loops.
The set $E$ of half edges is just the double of the set of edges,
and the edges $\{x, y\}\ (x \ne y)$ can be presented by an involution $\ta\colon x \leftrightarrow y$
acting freely on $E$ (note that we distinguish two ends of an edge even if it is a loop)
as the $\ang{\ta}$-orbits of $E$.
{\em We assume throughout the paper that each graph has at least one edge
and no isolated vertices
$($i.e., each vertex is connected to an edge$)$}.
Then since the set $E$ of half edges is the disjoint union of the form
\begin{equation}\tag{$*$}
E = \bigsqcup_{x \in V}\{e \in E \mid e \text{ is connected to }x\},
\end{equation}
the family of cyclic permutations of a Brauer graph can be seen
as the unique decomposition of a permutation $\si$ of $E$
into to the product of cyclic permutations (up to orderings), and hence this family itself is expressed by $\si$.
Finally, the set of vertices $V$ is regarded as the set $E/\si$ of $\ang{\si}$-orbits of $E$
because we have a bijection between $V$ and $E/\si$ by ($*$).
Note that a loop is expressed by a $\ang{\ta}$-orbit that is included in a $\ang{\si}$-orbit.
As a consequence in an abstract sense,
a Brauer graph can be seen just as a quadruple $(E, \si, \ta, m)$ of a set $E$,
a permutation $\si$ of $E$, an involution $\ta$ acting freely on $E$, and a map
$E/\si \to \bbN$,
which we call a {\em Brauer permutation}.
We always require that each $\ang{\si}$-orbit is a finite set to have a connection with
the multiplicity.
This formulation is convenient to compute coverings of Brauer graph algebras because
all necessary constructions can be expressed directly in terms of these four ingredients.
Moreover, it is easy to recover both the Brauer graph and the bound quiver
of the Brauer graph algebra corresponding to it as seen in Example \ref{exm:Br-gr-quiv}:
To obtain the Brauer graph shrink each $\ang{\si}$-orbit to one point, and
to obtain the bound quiver $(Q, I)$ define $Q$ by shrinking each $\ang{\ta}$-orbit to one point, and
then a set of generators of $I$ is given by $\si$ and $\ta$ automatically
(see Definition \ref{dfn:Br-gr-bdquiv}).

Now in \cite{GSS14} Green--Schroll--Snashall gave a way to compute
Galois coverings of Brauer graph algebras
with a finite abelian group $G$ which are again Brauer graph algebras.
This can be used to delete multiplicities, loops and multiple edges in
Brauer graphs and enables us to reduce problems on general Brauer graph algebras
to the corresponding problems on such Brauer graph algebras.
However, the description of the construction seems to be complicated.
To make it simple we introduced the notion of Brauer permutations explained above.
In Section 2 we apply the result in Section 1
to give a simple construction of  coverings of  a Brauer graph algebra
by a Brauer graph algebra/category in terms of Brauer permutations
without any assumptions on the group $G$
(Theorem \ref{thm:cover-Brauer}).
As applications we give ways of deleting multiplicities, loops, multiple edges,
and cycles in Brauer graphs
(Propositions \ref{prp:del-multi}, \ref{prp:del-loop}, \ref{prp:del-double}, and \ref{prp:del-cycle}).
The first three are already stated in \cite{GSS14},
but here we give their complete proofs and another much simpler unified deletion of loops.
The last one gives us a systematic way to delete all cycles (including loops, multiple edges)
in a Brauer graph to connect it to an infinite Brauer tree.

Finally in Section 3 we illustrate Propositions in Section 2 by some examples.

\section*{Acknowledgments}
This work was announced at Workshop on Brauer Graph Algebras held in March, 2016
in Stuttgart, and completed to write as a paper during my stay in Bielefeld in July--September, 2017.
I would like to thank Steffen K\"{o}nig, William Crawley-Boevey, Henning Krause, and Claus M.~Ringel
and all members of algebra seminars in both universities for their kind hospitality.
Finally, I also would like to thank the referee for his/her careful reading of the manuscript
and for some language corrections.

\section{smash products of bound quivers}

Let $Q = (Q_0, Q_1, s, t)$ be a quiver.
For each pair $(x, y)$ of vertices of $Q$ we denote by
$\bbP_Q(x,y)$ the set of all paths $\mu$ in $Q$ from $x$ to $y$ and set
$s(\mu):= x$ and $t(\mu):= y$.
We denote by $\k Q$ the $\k$-linear path category of $Q$
(also often regarded as the $\k$-linear path algebra when $Q$ is a finite quiver)
and by $\k Q^+$ the ideal of $\k Q$ generated by the set $Q_1$ of arrows.
An ideal $I$ of $\k Q$ is called {\em pre-admissible}
if $I$ is contained in $\k Q^+$ and
for each $x \in Q_0$ there exists a positive integer $m_x$ such that
all the paths of $Q$ with $x$ their source or target of length greater than
$m_x$ are contained in $I$.
(Note that $I$ is {\em admissible} if it is pre-admissible and
$I \le \k Q^{+2}$.)
A  {\em bound quiver} is a pair
$(Q, I)$ of a locally finite quiver $Q$ and a pre-admissible ideal $I$ of $\k Q$.
Note that $I$ is not assumed to be an admissible ideal of $\k Q$ in this paper.
This is because we want to include Brauer quiver algebras in considerations
whose relation ideals are pre-admissible but not always admissible.
Also note that if $(Q, I)$ is a bound quiver, then $\k(Q, I)$ is a locally bounded category
and that its Jacobson radical is given by $\k Q^+/I$.
(Since $I$ is not assumed to be admissible, $Q$ is not uniquely determined by $\k(Q, I)$.)
A {\em bound quiver morphism} $F\colon (Q, I) \to (Q', I')$ is a quiver morphism
$F\colon Q \to Q'$ of locally finite quivers $Q, Q'$
such that $\k F(I(x, y)) \le I'(Fx, Fy)$ for all $x, y \in Q_0$,
where $\k F$ is the $\k$-linearlization of $F$.

Throughout this section $A$ and $B$ are $G$-graded small $\k$-categories.

\begin{dfn}
We define two kinds of {\em smash products} $A \# G$ and $G \# A$ of $A$ and $G$.
\begin{enumerate}
\item
$A\# G$ is a $\k$-category with a right $G$-action defined as follows.
\begin{itemize}
\item
$(A \# G)_0:= \{x^{(a)}:= (x, a) \mid x \in A_0, a \in G\}$.
\item
$(A \# G)(x^{(a)}, y^{(b)}):= \{{}_{y^{(b)}}f_{x^{(a)}}:= (y^{(b)}, f, x^{(a)}) \mid f \in A^{ba\inv}(x, y) \}$, which is identified with $A^{ba\inv}(x, y)$ by the second projection
for all $x^{(a)}, y^{(b)} \in (A \# G)_0$.
\item
The composition of $A\# G$ is defined by the commutative diagram 
$$
\xymatrix{
(A \# G)(y^{(b)}, z^{(c)}) \times (A\# G)(x^{(a)}, y^{(b)}) &
(A \# G)(x^{(a)}, z^{(c)})\\
A^{cb\inv}(y, z) \times A^{ba\inv}(x, y) & A^{ca\inv}(x, z)
\ar"1,1";"1,2"
\ar@{=}"1,1";"2,1"
\ar@{=}"1,2";"2,2"
\ar"2,1";"2,2"
}
$$
for all $x^{(a)}, y^{(b)}, z^{(c)} \in (A\# G)_0$, where the morphism
in the bottom row is given by the composition of $A$;
namely, we have $${}_{z^{(c)}}g_{y^{(b)}}\circ {}_{y^{(b)}}f_{x^{(a)}}:= {}_{z^{(c)}}gf_{x^{(a)}}$$
for all $(g, f) \in A^{cb\inv}(y, z) \times A^{ba\inv}(x, y)$.
\item
$x^{(a)} \cdot c:= x^{(ac)}$ and
${}_{y^{(b)}}f_{x^{(a)}}\cdot c:= {}_{y^{(bc)}}f_{x^{(ac)}}$ 
for all $x^{(a)}, y^{(b)}\in (A\# G)_0, c\in G$, and
$f\in (A\# G)(x^{(a)}, y^{(b)}) = A^{ba\inv}(x, y) = (A\# G)(x^{(ac)}, y^{(bc)})$.
\end{itemize}
\item
$G\# A$ is a $\k$-category with a left $G$-action defined as follows.
\begin{itemize}
\item
$(G\# A)_0:= \{{}^{(a)}x:= (a, x) \mid a\in G, x\in A_0\}$.
\item
$
(G\# A)({}^{(a)}x, {}^{(b)}y):= \{
{}_{{}^{(b)}y}f_{^{(a)}x}:= (^{(b)}y, f, ^{(a)}x) \mid f \in A^{b\inv a}(x, y)
\}$,
which is identified with $A^{b\inv a}(x, y)$
for all ${}^{(a)}x, {}^{(b)}y \in (G\# A)_0$
by the second projection.

\item
The composition of $G\# A$ is defined by the commutative diagram 
$$
\xymatrix{
(G\# A)({}^{(b)}y, {}^{(c)}z) \times (G\# A)({}^{(a)}x, {}^{(b)}y) &
(G \# A)({}^{(a)}x, {}^{(c)}z)\\
A^{c\inv b}(y, z) \times A^{b\inv a}(x, y) & A^{c\inv a}(x, z)
\ar"1,1";"1,2"
\ar@{=}"1,1";"2,1"
\ar@{=}"1,2";"2,2"
\ar"2,1";"2,2"
}
$$
for all ${}^{(a)}x, {}^{(b)}y, {}^{(c)}z \in (G\# A)_0$, where the morphism
in the bottom row is given by the composition of $A$.
\item
$c\cdot{}^{(a)}x:= {}^{(ca)}x$, and
$c\cdot f:= f$  for all ${}^{(a)}x, {}^{(b)}y\in G\# A, c\in G$
and $f\in (G\# A)({}^{(a)}x, {}^{(b)}y) = A^{b\inv a}(x, y) = (G\# A)({}^{(ca)}x, {}^{(cb)}y)$.
\end{itemize}
\end{enumerate}
\end{dfn}

\begin{rmk}
\label{rmk:free-action-can-cov}
(1) Note that both $G$-actions defined above are free actions.

(2) Define a  functor $F\colon A \# G \to A$ by 
$$
\begin{aligned}
x^{(a)} &\mapsto x \ \text{and}\\
(A \# G)(x^{(a)}, y^{(b)}) &= A^{ba\inv}(x, y) \incl A(x, y) = A(F(x^{(a)}), F(y^{(b)}))
\end{aligned}
$$
for all $x^{(a)}, y^{(b)} \in (A\# G)_0$.
Then $F$ is a (strict) $G$-covering in the sense of \cite{Asa11}.
If $A$ is a locally bounded category then $F$ is a Galois covering with group $G$.
Indeed, if we denote by $(A \# G)/\!_g\, G$ the generalized orbit category of $A \# G$ by $G$
defined in \cite{Asa11} and by $(A \# G)/ G$ the usual orbit category defined in \cite{Ga81},
then we have equivalences $(A \# G)/ G \simeq (A \# G)/\!_g\, G \simeq A$.
The first equivalence follows by the fact that the $G$-action is free,
and the second is shown in papers such as \cite{CM06, Asa11, Asa15}.
\end{rmk}

\begin{lem}
Set $(A\op)^a(x, y):= A^{a\inv}(y, x)$ for all $a\in G, x, y\in A_0$.
Then this defines a $G$-grading on $A\op$.
We usually regard $A\op$ as a $G$-graded category
with this $G$-grading.
\end{lem}

\begin{proof}
Straightforward.
\end{proof}

\begin{lem}
We have an isomorphism
$G\# A \iso (A\op \# G)\op$.
\end{lem}

\begin{proof}
Straightforward.
\end{proof}

In papers published before we dealt with the smash product $G\# A$ (with the notation
$A\# G$), but
in the sequel we use the terminology ``the smash product of $A$ and $G$''
to mean the smash product $A\# G$ defined above.

\begin{lem}\label{lem:factor-alg-smash}
Let $A$ be a $G$-graded category and $I$ a homogeneous ideal of $A$.
For each $x^{(a)}, y^{(b)} \in (A\# G)_0$ $(x, y \in A_0, a, b \in G)$
define $(I\#G)(x^{(a)}, y^{(b)}):= I^{ba\inv}(x, y)$ be a $\k$-subspace of $(A\#G)(x^{(a)}, y^{(b)})$.
Then $I\# G$ turns out to be an ideal of $A\# G$ and we have
a natural isomorphism
$$
(A\# G)/(I\# G) \iso (A/I)\# G
$$
as categories with right $G$-actions, by which we identify these categories.
\end{lem}

\begin{proof}
Let $x^{(a)}, x'^{(a')}, y^{(b)}, y'^{(b')} \in (A\# G)_0$.
Then
$$
\begin{aligned}
&\ (A\# G)(y^{(b)}, y'^{(b')})\cdot (I\# G)(x^{(a)}, y^{(b)})\cdot (A\# G)(x'^{(a')}, x^{(a)})\\
=&\ A^{b'b\inv}(y, y')\cdot I^{ba\inv}(x,y) \cdot A^{a{a'}\inv}(x', x)
\le  I^{b'{a'}\inv}(x', y') = (I\# G)(x'^{(a')}, y'^{(b')}).
\end{aligned}
$$
Therefore $I\# G$ is an ideal of $A\# G$.

Next the object sets of both hand sides coincide.
Indeed,
$$
((A\# G)/(I\# G))_0 = (A\# G)_0 = A_0 \times G = (A/I)_0 \times G
= ((A/I) \# G)_0.
$$

Moreover, morphism spaces of both hand sides coincide.
Indeed,
\begin{equation}\label{eq:identify}
\begin{aligned}
((A\# G)/(I\# G))(x^{(a)}, y^{(b)}) &= (A\# G)(x^{(a)}, y^{(b)})/(I\# G)(x^{(a)}, y^{(b)})\\
&= A^{ba\inv}(x, y)/I^{ba\inv}(x, y)\\
&\iso (A^{ba\inv}(x, y) + I(x,y))/I(x, y)\\
&= (A/I)^{ba\inv}(x, y)\\
&= ((A/I)\# G)(x^{(a)}, y^{(b)}).
\end{aligned}
\end{equation}
We may identify both morphism spaces because the isomorphism in \eqref{eq:identify} is natural.

In addition, the compositions of both hand sides coincide.
Indeed,
we have the commutative diagram
$$
\xymatrix{
((A\# G)/(I\# G))(y^{(b)}, z^{(c)}) \times ((A\# G)/(I\# G))(x^{(a)}, y^{(b)}) &
((A\# G)/(I\# G))(x^{(a)}, z^{(c)})\\
(A/I)^{cb\inv}(y, z) \times (A/I)^{ba\inv}(x, y) & (A/I)^{cd\inv}(x, z)\\
((A/I)\# G)(y^{(b)}, z^{(c)}) \times ((A/I)\# G)(x^{(a)}, y^{(b)}) &
((A/I)\# G)(x^{(a)}, z^{(c)}),
\ar"1,1";"1,2"
\ar"2,1";"2,2"
\ar"3,1";"3,2"
\ar@{=}"1,1";"2,1"
\ar@{=}"2,1";"3,1"
\ar@{=}"1,2";"2,2"
\ar@{=}"2,2";"3,2"
}
$$
where the horizontal maps are given by
the compositions of the corresponding categories.

Finally, the right $G$-actions on both sides coincide.
Indeed,
the coincidence on objects are trivial, and for each
$(x, a), (y,b) \in A_0 \times G$, $f \in A^{ba\inv}(x, y)$ and $c \in G$
the action of $c$ on the left and right hand sides are given by
$$
\begin{aligned}
{}_{y^{(b)}}f_{x^{(a)}} + (I\# G)(x^{(a)}, y^{(b)}) &\mapsto
{}_{y^{(bc)}}f_{x^{(ac)}} + (I\# G)(x^{(ac)}, y^{(bc)})
\ \text{and}\\
{}_{y^{(b)}}(f + I(x, y))_{x^{(a)}} &\mapsto
{}_{y^{(bc)}}(f + I(x, y))_{x^{(ac)}},
\end{aligned}
$$
respectively.
These actions coincide by our identification  in \eqref{eq:identify}.
\end{proof}

\begin{dfn}
\label{dfn:min-rel-hmg}
Let $(Q,I)$ be a bound quiver.
Then a map $W\colon Q_1 \to G$ is called a $G$-{\em weight} on $Q$.
\begin{enumerate}
\item
An element $\ro = \sum_{i=1}^n t_i\mu_i$ of $I$
$(t_i \in \k, \mu_i \text{ are parallel paths of }Q)$ is called a
{\em minimal relation} if $\sum_{i\in J}t_i \mu_i \not\in I$ for every
proper subset $J$ of $\{1, \dots, n\}$.
\item
For each path $\mu = \al_n\dots \al_1$ of $Q$ we set
$$
W(\mu):= W(\al_n)\dots W(\al_1).
$$
\item
$W$ is called a {\em homogeneous weight} on $(Q, I)$ if
for each minimal relation $\sum_{i=1}^n t_i\mu_i \in I$
we have
$$
W(\mu_i) = W(\mu_1)
$$
for all $i = 1, \dots, n$. 
\end{enumerate}
\end{dfn}

\begin{exm}
\label{exm:br1}
Let $G$ be the additive group $\bbZ$,
$Q$ the quiver
$$
\vcenter{\xymatrix{
2\\
1\\
3
\ar@/^/"1,1";"2,1"^{\al_2}
\ar@/^/"2,1";"1,1"^{\al_1}
\ar@/_/"2,1";"3,1"_{\be_1}
\ar@/_/"3,1";"2,1"_{\be_2}
}},
$$
$I$ the ideal $\ang{\al_2\al_1 - \be_2\be_1, \be_1\al_2, \al_1\be_2,
\al_2\al_1\al_2, \be_2\be_1\be_2}$ of $\k Q$, and $W$ the $G$-weight on $Q$
defined by $W(\al_1) = 0 = W(\be_1), W(\al_2) = 1 = W(\be_2)$.
Then $W$ is a homogeneous weight on $(Q, I)$.
\end{exm}

\begin{rmk}
Assume that $W$ is a homogeneous weight on $(Q, I)$.
Then 
\begin{enumerate}
\item
$I$ is a homogeneous ideal of the $G$-graded $\k$-category $\k Q$, where
the $G$-grading is given by
$$
(\k Q)^a(x, y):= \Ds_{\smat{\mu\in Q_{\ge 0}(x, y)\\W(\mu)=a}}\k \mu
$$
for all $a \in G, x, y \in Q_0$.
\item
We set $\k(Q, I, W)$ to be the $G$-graded category $\k Q/I$ with the
$G$-grading given by
$$
(\k Q/I)^a(x, y):= ((\k Q)^a(x, y) + I(x, y))/I(x, y)
$$
for all $a \in G, x, y \in Q_0$.
\end{enumerate}
\end{rmk}

\begin{dfn}\label{dfn:cov-by-weight}
Let $Q$ be a quiver, $I$ an ideal of the category $\k Q$ contained in $\k Q^+$,
and $W$ a homogeneous $G$-weight on $\k(Q, I)$.
Define a quiver
$$
Q_{G,W} = ((Q_{G,W})_0, (Q_{G,W})_1, s_{G,W}, t_{G,W})
$$
and an ideal $I_{G,W}$ of $\k Q_{G,W}$ as follows.
$$
\begin{aligned}
(Q_{G,W})_0&:= \{x^{(a)}:= (x, a) \mid x \in Q_0, a \in G\}=Q_0 \times G\\
(Q_{G,W})_1&:= \{\al^{(a)}\colon x^{(a)} \to y^{(W(\al)a)} \mid \al \colon x \to y \text{ in } Q_1, a \in G\}\\
(\text{thus }s_{G,W}(\al^{(a)})&:= x^{(a)},\ \ t_{G,W}(\al^{(a)}):= y^{(W(\al)a)} \text{ for }\al \colon x \to y \text{ in } Q_1, a \in G),\\
I_{G,W}&:=\ang{\ro^{(a)} \mid a \in G, 
\ro \text{ is a minimal relation in }I},
\end{aligned}
$$
where for each path $\mu= \al_n\cdots \al_1$ of length $n \ge 2$ and $a \in G$, we set $\mu^{(a)}$ to be the path
$$
\mu^{(a)}:=\al_n^{(a_{n-1}\cdots a_1 a)}\cdots\al_2^{(a_1 a)}\al_1^{(a)}
$$
from $x^{(a)}$ to $y^{(W(\mu) a)}$
with $a_i:= W(\al_i)$ for all $i=1, \dots, n$ and $W(\mu):= W(\al_n)\cdots W(\al_1)$,
and for each element $\ro = \sum_{i}k_i\mu_i \ (k_i \in \k, \mu_i\text{: paths})$ of $I(x, y)\ (x, y \in Q_0)$ we set
$$
\ro^{(a)}:= \sum_{i}k_i\mu_i^{(a)}.
$$
Then $\k(Q_{G,W}, I_{G,W})$ is a $\k$-category with a right $G$-action $X$ defined by
the quiver morphism
$$
X_c\colon (x^{(a)} \ya{\al^{(a)}}y^{(b)}) \mapsto (x^{(ac)}\ya{\al^{(ac)}}y^{(bc)})
$$
for all $x^{(a)} \ya{\al^{(a)}}y^{(b)}$ in $(Q_{G,W})_1$
($a, c \in G, x, y\in Q_0, \al \in Q_1, b= W(\al)a$).
We 
call $(Q_{G,W}, I_{G,W})$ the {\em smash product} of $(Q, I, W)$ and $G$.
\end{dfn}

\begin{rmk}
\label{rmk:weight}
For $b \in G$ and $\al \colon x \to y$ in $Q_1$ we have an arrow
$\al^{(W(\al)\inv b)}\colon x^{(W(\al)\inv b)} \to y^{(b)}$
in $Q_{G,W}$.
\end{rmk}

Here we recall the definitions of coverings of quivers and of bound quivers.

\begin{dfn}
\label{dfn:covering}
Let $Q = (Q_0, Q_1, s, t)$ be a quiver.
\par
(1) For each $x \in Q_0$ we set
$$
\begin{aligned}
x^+&:= \{ \al \in Q_1 \mid s(\al) = x\}\ \text{and}\\
x^-&:= \{ \al \in Q_1 \mid t(\al) = x\}.
\end{aligned}
$$
\par
(2) Paths $\mu_1, \mu_2, \dots, \mu_n\  (n \ge 2)$ in $Q$ are said to be {\em parallel}
if $s(\mu_i) = s(\mu_1)$ and $t(\mu_i) = t(\mu_1)$ for all $i = 1, 2, \dots, n$.

\medskip

 Let $Q' = (Q'_0, Q'_1, s', t')$ be another quiver and
 $F \colon Q \to Q'$ a morphism of quivers.

\medskip

(3) $F$ is called a {\em covering} of quivers if it is surjective on the vertices and induces bijections
$$
x^+ \to (Fx)^+\quad\text{and}\quad x^- \to (Fx)^-
$$
for all $x \in Q_0$.
Moreover, a covering $F$ of quivers is called {\em regular} if
$F_*(\pi_1(Q, x_0))$ is a normal subgroup of $\pi_1(Q', F(x_0))$,
where $\pi_1(R, x)$ is the fundamental group of a quiver $R$ with a base point $x \in R_0$,
and $F_* \colon \pi_1(Q, x_0) \to \pi_1(Q', F(x_0))$ is the map canonically induced from $F$.
\par
(4) A map $L\colon Q'_0 \to Q_0$ is called a {\em lifting} of $F$ if $FL = \id_{Q'_0}$.

(5) By definition of a covering of quivers note that if $L$ is a lifting of a covering $F$,
then for any path (or even any walk) $\mu = \be_n\cdots \be_2\be_1$ in $Q'$
there exists a unique path $\la = \al_n \dots\al_2\al_1$ in $Q$
such that $s(\la) = L(s'(\mu))$
and $F(\la) = \mu$,
where we set $F(\la):= F(\al_n)\cdots F(\al_2)F(\al_1)$.
We then set $L(\mu):= \la$.
For each $x \in Q'_0$ the linearlization $\k Q'(x, \blank) \to \k Q(Lx, \blank)$
of $L$ is denoted also by $L$.
\par
(6) Let $(Q, I)$ and $(Q', I')$ be bound quivers and $F\colon (Q, I) \to (Q', I')$
a morphism of bound quivers.
Then $F$ is called a {\em covering} of bound quivers if
$F \colon Q \to Q'$ is a regular covering of quivers and the following are satisfied:
\begin{enumerate}
\item[(a)]
For each minimal relation $\ro = \sum_{i=1}^n k_i \mu_i$ in $I'$
($k_i \in \k$, $\mu_i$ are parallel paths in $Q'$) and each lifting $L$ of $F$
all paths $L(\mu_i)$ are parallel in $Q$; and
\item[(b)]
$I = \{ L(\ro) \mid \ro \in I', L \text{ is a lifting of }F\}$.
\end{enumerate}
\end{dfn}

\begin{rmk}
In Definition \ref{dfn:covering}(3) note that $x^+$ and $x^-$ are subsets of arrows,
which is different from the usual definition in \cite{Rie80}.
This is because we allow quivers to have multiple arrows here.
\end{rmk}

\begin{dfn}
Let $(Q, I)$ be a bound quiver with $Q = (Q_0, Q_1, s, t)$.

(1) We denote by $\Aut (Q, I)$ the group of automorphisms of the bound quiver $(Q, I)$.

(2) A (right) $G$-{\em action} on $(Q, I)$ is a group homomorphism $X\colon G\op \to \Aut(Q,I)$.
We denote $X(a)x$ simply by $xa$ for all $a \in G$ and $x \in Q_0 \cup Q_1$
if there seems to be no confusion.
We also set $xG:= \{xa \mid a \in G\}$ for all $x \in Q_0 \cup Q_1$.

Let $X$ be a $G$-action on $(Q, I)$.

(3) The orbit quiver $Q/G$ is the quiver $(Q_0/G, Q_1/G, \ovl{s}, \ovl{t})$,
where $Q_i/G:= \{xG \mid x \in Q_i\}\ (i = 0, 1)$ and for each $r \in \{s, t\}$,
$\ovl{r}$ is the map $Q_1/G \to Q_0/G$ defined by
$\ovl{r}(\al G):= r(\al)G$\ ($\al \in Q_1$), which
is well-defined because $r$ is commutative with $X(a) \in \Aut(Q, I)$ for all $a \in G$.

(4) A quiver morphism $\pi \colon Q \to Q/G$ is defined by
$x \mapsto xG,\ (x \in Q_0 \cup Q_1)$, which is called the {\em canonical morphism}.

(5) $X$ is said to be {\em admissible} if
for each $x \in Q_0$,
$\al \ne \be$ implies $G\al \ne G\be$ for all $\al, \be \in x^+$ and for all $\al, \be \in x^-$,
or equivalently,
for each $p \in \{+, -\}$,
the $G$-orbit of each arrow intersects with $x^p$ at most once.

(6) $X$ is said to be {\em free} if 
$xa \ne x$ for all $1 \ne a \in G$ and $x \in Q_0$.
\end{dfn}

\begin{lem}\label{lem:adm-free}
Let $(Q, I)$ be a bound quiver with a $G$-action.
Then the following hold.
\begin{enumerate}
\item
The canonical morphism $\pi\colon Q \to Q/G$ turns out to be a covering of quivers if and only if
the $G$-action is admissible.
If this is the case, then
\begin{enumerate}
\item
the canonical morphism $\pi\colon Q \to Q/G$ turns out to be
a regular covering morphism; and
\item
$\pi$ induces the canonical
covering morphism $\ovl{\pi} \colon (Q, I) \to (Q, I)/G$ of bound quivers, where
we set $(Q, I)/G:= (Q/G, (\k \pi)(I))$.
\end{enumerate}
\item
Assume that the $G$-action on $(Q, I)$ is admissible and free.
Then
$\k \pi\colon Q \to Q/G$ induces a Galois covering functor
$\k\ovl{\pi} \colon \k(Q, I) \to \k((Q, I)/G)$ with group $G$.
\end{enumerate}
\end{lem}

\begin{proof}
(1) This immediately follows by Definition \ref{dfn:covering}(3).
It is straightforward to check (a) and (b).

(2) It is easy to see that
$\k\ovl{\pi}$ 
is a covering functor.
Finally,  $\k\ovl{\pi}$ is a Galois covering functor with group $G$ because
we have $\k\ovl{\pi} \cdot X(a) = \k\ovl{\pi}$ for all $a \in G$, $\k\ovl{\pi}$ is surjective on the objects, and
$G$ acts transitively on the fibers $xG = \pi\inv(xG)$ for all $x \in Q_0$
(see \cite[3.1 Remark]{Ga81}).
\end{proof}

\begin{dfn}\label{dfn:Galois-cov-mor}
A bound quiver morphism $E \colon (Q, I) \to (R, J)$ is called a {\em Galois covering} morphism with group $G$
if it is isomorphic to the canonical covering $\ovl{\pi}\colon (Q, I) \to (Q, I)/G$ given by an admissible and free
$G$-action on $(Q, I)$,
namely if there exists an isomorphism $H\colon (Q, I)/G \to (R, J)$ such that
$E = H\ovl{\pi}$.
\end{dfn}

\begin{lem}
Let $E \colon (Q, I) \to (R, J)$ be a Galois covering morphism with group $G$
between bound quivers.
Then the induced functor $\k E \colon \k(Q, I) \to \k(R, J)$ is a Galois covering with group $G$.
\end{lem}

\begin{proof}
Take the $G$-action on $\k(Q, I)$ defined by the $G$-action on $(Q, I)$,
and assume that there exists an isomorphism $H\colon (Q, I)/G \to (R, J)$ such that
$E = H\ovl{\pi}$ as in Definition \ref{dfn:Galois-cov-mor}.
Then clearly $\k E = \k H\cdot \k \ovl{\pi}$ and $\k H$ is an isomorphism,
and here $\k \ovl{\pi}$ is a Galois covering with group $G$ by Lemma \ref{lem:adm-free}(2).
\end{proof}

\begin{prp}
\label{prp:can-cover}
The morphism of bound quivers $F_{G,W} \colon (Q_{G,W}, I_{G,W}) \to (Q,I)$ defined by
$F_{G,W}(x^{(a)}):= x, F_{G,W}(\al^{(a)}):= \al$ $(x \in Q_0, \al \in Q_1, a \in G)$ is
a Galois covering with group $G$.
\end{prp}

\begin{proof}
Let $x, y \in Q_0$ (resp.\ $x, y \in Q_1$) and $a, b \in G$.
Then $x^{(a)}G = y^{(b)}$ if and only if
$y^{(b)} = x^{(a)}\cdot c$ for some $c \in G$
if and only if $y = x$ and $b = ac$ for some $c \in G$
if and only if $y = x$.
This means that the correspondence
$x^{(a)}G \mapsto x$ defines a map $H_0 \colon (Q_{G,W}/G)_0 \to Q_0$
(resp.\  $H_1\colon (Q_{G,W}/G)_1 \to Q_1$) and that both are injective.
Obviously these maps are surjective.
Hence $H:= (H_0, H_1) \colon Q_{G,W}/G  \to Q$ is an isomorphism of quivers.
A direct calculation shows that $F_{G,W} = H \pi$.

It remains to show that $\k H(\k \pi(I_{G,W})) = I$, or equivalently that
$\k F_{G,W}(I_{G,W}) = I$.
To this end it is enough to show the following:
$$
I_{G,W} = \{ L(\ro) \mid \ro \in I, L \text{ is a lifting of }F\}.
$$
($\subseteq$). As is easily seen the right hand side is an ideal of $\k Q_{G,W}$.
For each $a \in G$ define a map $L_a\colon Q_0 \to (Q_{G,W})_0$ by
$L_a(x):= x^{(a)}\ (x \in Q_0)$.
Then $L_a$ is a lifting of $F$, and $L_a(\mu) = \mu^{(a)}$ for all $a \in G$
and path $\mu$ in $Q$.
Hence the generating set of $I_{G,W}$ is included in the right hand side,
which shows the inclusion $(\subseteq)$.
\\
($\supseteq$).
Let $x, y \in Q_0$, $\ro \in I(x, y)$ and $L$ be a lifting of $F$.
Then $\ro = \ro_1 +\cdots + \ro_t$ for some minimal relations
$\ro_1,\dots, \ro_t \in I(x, y)$.
There exists an $a \in G$ such that $L(x) = x^{(a)}$.
With this $a$ we have
$L(\ro) = \ro^{(a)} = \ro_1^{(a)}+\cdots + \ro_t^{(a)} \in I_{G,W}$. 
\end{proof}

\begin{thm}
\label{thm:quiv-pres-sm-prod}
Let $(Q, I)$ be a bound quiver and $W$ a homogeneous $G$-weight on $\k(Q, I)$.
Then the smash product $\k(Q, I, W)\# G$ is presented by the bound quiver
$(Q_{G,W},I_{G,W})$, i.e., we have an isomorphism
$$
\k(Q, I, W)\# G\iso \k(Q_{G,W},I_{G,W})
$$
as $\k$-categories with right $G$-actions.
\end{thm}

\begin{proof}
Note by Lemma \ref{lem:factor-alg-smash} that we have
$$
\k(Q, I, W)\# G = (\k(Q, W)\# G)/(I \# G).
$$
We first define a functor $\ph\colon \k(Q,W)\# G \to \k(Q_{G,W})$.
\\
{\bf On objects.}  For each $x^{(a)} \in (\k(Q,W)\# G)_0\ (x \in Q_0, a\in G)$,
we set
$$
\ph(x^{(a)}):= x^{(a)} \in (Q_{G,W})_0.
$$
\\
{\bf On morphisms.}
Let $x^{(a)}, y^{(b)} \in (\k(Q, W)\# G)_0$ and $f \in (\k(Q,W)\# G)(x^{(a)}, y^{(b)}) = \k(Q,W)^{ba\inv}(x,y)$.
Then $f = \sum_{i=1}^{n}k_i\mu_i$ for some $k_i \in \k, \mu_i \in \bbP_Q(x,y)$
with $W(\mu_i) = ba\inv$.
Here note that each $\mu_i^{(a)}$ is a path from $x^{(a)}$ to $y^{(b)}$ by Remark \ref{rmk:weight}.
Then we set
$$
\ph(f):= \sum_{i=1}^n k_i \mu_i^{(a)}\in \k Q_{G,W}(x^{(a)}, y^{(b)}).
$$
\subsection*{(1) $\ph$ is a $\k$-functor.}
Let $x^{(a)} \in (\k(Q, W)\# G)_0$.
Then
$$
\id_{x^{(a)}}\in (\k(Q, W)\# G)(x^{(a)}, x^{(a)}) = \k(Q, W)^1(x,x),
$$
which is a linear combination of paths from $x$ to $x$,
and hence $\id_{x^{(a)}} = \id_{x^{(a)}}e_x = e_x$.
Therefore $\ph(\id_{x^{(a)}}) = \ph(e_x) = e_x^{(a)} = \id_{x^{(a)}}$ in $\k(Q_{G,W}, I_{G,W})$.

Let $x^{(a)} \ya{f} y^{(b)} \ya{g} z^{(c)}$ in $\k(Q, W) \# G$.
Then $f = \sum_{i=1}^mk_i \la_i$, $g =\sum_{j=1}^n l_j \mu_j$
for some $k_i, l_j \in \k$, $\la_i \in \bbP_Q(x,y), \mu_j \in \bbP_Q(y,z)$
with $W(\la_i) = ba\inv, W(\mu_j) = cb\inv$.
Then it is easy to see that $(\mu_j \la_i)^{(a)} = \mu_j^{(b)}\la_i^{(a)}$
for all $i, j$, which shows
$$
\begin{aligned}
\ph(g\cdot f) &= \ph(\sum_{i,j}(l_j k_i)\mu_j \la_i) =
\sum_{i,j}(l_j k_i)(\mu_j\la_i)^{(a)}\\
&= \sum_{j=1}^nl_j\mu_j^{(b)}\sum_{i=1}^m k_i\la_i^{(a)}
= \ph(g)\cdot \ph(f).
\end{aligned}
$$
It is obvious that $\ph$ is $\k$-linear.

\subsection*{(2) $\ph(I\# G) \subseteq I_{G, W}$}
Thus $\ph$ induces a functor $\ovl{\ph}\colon \k(Q,I,W)\# G \to \k(Q_{G,W}, I_{G,W})$.

Let $x^{[a]}, y^{[b]} \in (A\# G)_0$.
We have only to show that
$\ph((I\# G)(x^{[a]}, y^{[b]})) \subseteq I_{G, W}(x^{[a]}, y^{[b]})$.
Let $f \in  (I\# G)(x^{(a)}, y^{(b)})  = I^{ba\inv}(x,y)$.
Then $f = \sum_{i=1}^m k_i \mu_i$ for some $k_i \in \k$ and parallel paths $\mu_i$
with $W(\mu_i) = ba\inv$, and we have
$f = \ro_1 + \ro_2 +\cdots +\ro_t$ for some minimal relations $\ro_j$ in $I$,
where for each $j = 1, 2,\dots, t$ we have $\ro_j = \sum_{i \in S_j} k_i \mu_i$
for some partition $S_1 \sqcup S_2 \sqcup \dots \sqcup S_t = \{1, 2, \dots, n\}$
of the set $\{1, 2, \dots, n\}$.
Now $\ph(f) = \sum_{i=1}^m k_i \mu_i^{(a)} = \sum_{j=1}^t\sum_{i\in S_j}k_i\mu_i^{(a)}
= \ro_1^{(a)} + \ro_2^{(a)} + \dots + \ro_t^{(a)} \in I_{G, W}$.
\subsection*{(3) $\ovl{\ph}$ is a bijection on the objects.}
This is trivial because $\ph$ is the identity on objects.

\subsection*{(4) $\ovl{\ph}$ commutes with right $G$-actions}
Let $x, y \in Q_0$ and $a,b,c \in G$.
We have to show the commutativity of the diagram
\begin{equation}
\label{eq:right-G-action}
\vcenter{
\xymatrix{
(\k(Q,I,W)\# G)(x^{(a)},y^{(b)}) & \k(Q_{G,W}, I_{G,W})(x^{(a)},y^{(b)})\\
(\k(Q,I,W)\# G)(x^{(ac)},y^{(bc)}) & \k(Q_{G,W}, I_{G,W})(x^{(ac)},y^{(bc)}).
\ar^{\ovl{\ph}}"1,1";"1,2"
\ar_{\ovl{\ph}}"2,1";"2,2"
\ar_{(\blank)^c}"1,1";"2,1"
\ar^{X_c}"1,2";"2,2"
}
}
\end{equation}
It is enough to show the commutativity for each element of the form
$\ovl{\mu}:= \mu + (I\# G)(x^{(a)},y^{(b)})$ for a path $\mu \in \bbP_Q(x, y)$ with $W(\mu) = ba\inv$.
This is verified by the equalities
$$
\ovl{\ph}(({}_{y^{(b)}}\ovl{\mu}_{x^{(a)}})^c)
= \ovl{\ph}({}_{y^{(bc)}}\ovl{\mu}_{x^{(ac)}})
= \widetilde{\mu^{(ac)}}
=X_c(\widetilde{\mu^{(a)}})
=X_c(\ovl{\ph}({}_{y^{(b)}}\ovl{\mu}_{x^{(a)}})),
$$
where $\widetilde{(\blank)}$ stands for the coset in $\k(Q_{G,W}, I_{G,W})$.

\subsection*{(5) $\ovl{\ph}$ is fully faithful}
Let $x^{(a)}, y^{(b)} \in (\k(Q, I, W)\# G)_0$.
Then we have a commutative diagram
$$\footnotesize
\xymatrix{
0 & (I\# G)(x^{(a)}, y^{(b)}) & (\k(Q,W)\# G)(x^{(a)}, y^{(b)}) & (\k(Q,I,W)\# G)(x^{(a)}, y^{(b)}) & 0\\
0 & I_{G,W}(x^{(a)}, y^{(b)}) & \k Q_{G,W}(x^{(a)}, y^{(b)}) & \k(Q_{G,W}, I_{GW})(x^{(a)}, y^{(b)}) & 0
\ar"1,1";"1,2"
\ar@{^{(}->}"1,2";"1,3"
\ar"1,3";"1,4"
\ar"1,4";"1,5"
\ar"2,1";"2,2"
\ar@{^{(}->}"2,2";"2,3"
\ar"2,3";"2,4"
\ar"2,4";"2,5"
\ar"1,2";"2,2"_{\ph|I\# G}
\ar"1,3";"2,3"^\ph
\ar"1,4";"2,4"^{\ovl{\ph}}
}
$$
with exact rows.
Therefore it is enough to show that both $\ph$ and $\ph|I\# G$ above are isomorphisms
by 5-Lemma.

First we show that
$\ph \colon (\k(Q,W)\# G)(x^{(a)}, y^{(b)}) \to \k Q_{G,W}(x^{(a)}, y^{(b)})$ is
an isomorphism.
For each $c \in G$ we set
$\bbP_Q^c(x,y):= \{ \mu \in \bbP_Q(x, y) \mid W(\mu) = c\}$.
Then 
 $(\k(Q,W)\# G)(x^{(a)}, y^{(b)}) = \k(Q, W)^{ba\inv}(x,y)$ has a basis
 $\bbP_Q^{ba\inv}(x, y)$, while the space
 $\k Q_{G,W}(x^{(a)}, y^{(b)})$ has a basis $\bbP_{Q_{G,W}}(x^{(a)}, y^{(b)})$,
and $\ph$ induces a map
 $$
 \ph_0\colon \bbP_Q^{ba\inv}(x, y) \to \bbP_{Q_{G,W}}(x^{(a)}, y^{(b)}),
 \mu \mapsto \mu^{(a)}.
 $$
 Hence it suffices to show that $\ph_0$ is bijective.
 Let $F:=F_{G,W}\colon Q_{G,W} \to Q$ be the covering defined in Proposition \ref{prp:can-cover}.
 This induces a map
 $F_0\colon\bbP_{Q_{G,W}}(x^{(a)}, y^{(b)}) \to \bbP_Q^{ba\inv}(x, y)$.
 We show that $\ph_0$ and $F_0$ are inverses to each other, which will prove that
 $\ph_0$ is bijective.
 For each $\mu \in \bbP_Q^{ba\inv}(x, y)$ we have
 $F_0(\ph_0(\mu)) = F(\mu^{(a)}) = \mu$, which shows that
 $F_0\ph_0 = \id_{\bbP_Q^{ba\inv}(x, y)}$.
 Let $\xi \in \bbP_{Q_{G,W}}(x^{(a)}, y^{(b)})$ and set
 $\mu:= F(\xi)$.
 Then $F(\mu^{(a)}) = \mu = F(\xi)$ and
 $s_{G,W}(\mu^{(a)}) = x^{(a)} = s_{G,W}(\xi)$.
 Therefore by the uniqueness of lifting (see Definition \ref{dfn:covering} (5)) we have $\mu^{(a)} = \xi$,
 and $\ph_0(F_0(\xi)) = \mu^{(a)} = \xi$, which shows that
 $\ph_0 F_0 = \id_{\bbP_{Q_{G,W}}(x^{(a)}, y^{(b)})}$.
 
 Next we show that $\ph|I\# G$ is an isomorphism.
 By the commutativity of the left square $\ph|I\# G$ is injective because so is $\ph$ above.
 Now let $\ro^{(a)} \in I_{G,W}$ with $a \in G$ and $\ro$ a minimal relation in $I(x, y)$
 with $x, y \in Q_0$.
 Then we have $s_{G,W}(\ro^{(a)}) = x^{(a)}$ and $t_{G,W}(\ro^{(a)}) = y^{(b)}$
 for some $b \in G$.
 Set $\ro = \sum_{i=1}^m k_i\mu_i$ for some
 $0 \ne k_i \in \k$ and $\mu_i \in \bbP_Q(x, y)$.
 Here since $W$ is a homogeneous weight, we have $W(\mu_i) = W(\mu_1) = ba\inv$.
 Then $\ro \in I^{ba\inv}(x, y) = (I\# G)(x^{(a)}, y^{(b)})$ and
 $\ro^{(a)} = \ph(\ro) \in \ph((I\# G)(x^{(a)}, y^{(b)}))$.
Therefore $\ph|I\# G$ is surjective, and hence an isomorphism.
 \end{proof}

\begin{rmk}
As explained in the introduction
we can obtain the statement of Theorem \ref{thm:quiv-pres-sm-prod} indirectly by
combining theorems obtained in \cite{Gr83}
under the assumption that $Q_{G,W}$ is connected.
The proof above is a direct one and does not require the connectedness assumption.
\end{rmk}

\begin{exm}
Let $G = \bbZ$ and $(Q, I, W)$ be the bound quiver with weight defined in
Example \ref{exm:br1}.
Then the smash product $\k(Q, I, W)\# G$ is given by the bound quiver
$(Q_{G, W}, I_{G, W})$, where $Q_{G, W}$ is the quiver
$$
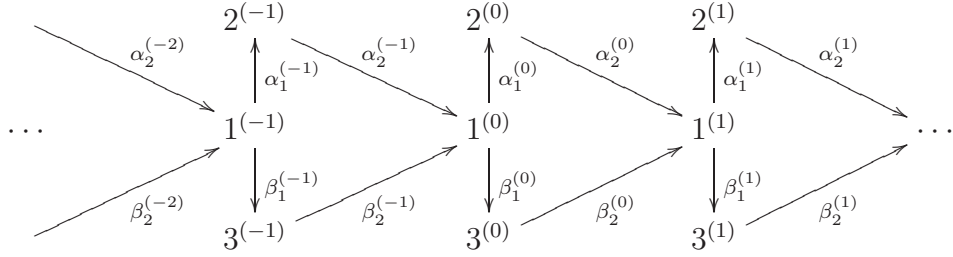
\vcenter{\xymatrix@C=60pt{
& 2^{(-1)} & 2^{(0)} & 2^{(1)}\\
\cdots & 1^{(-1)} & 1^{(0)} & 1^{(1)} & \cdots\\
& 3^{(-1)} & 3^{(0)} & 3^{(1)}\\
\ar"1,1";"2,2"^{\al_2^{(-2)}}
\ar"3,1";"2,2"_{\be_2^{(-2)}}
\ar"2,2";"1,2"_{\al_1^{(-1)}}
\ar"2,2";"3,2"^{\be_1^{(-1)}}
\ar"2,3";"1,3"_{\al_1^{(0)}}
\ar"2,3";"3,3"^{\be_1^{(0)}}
\ar"2,4";"1,4"_{\al_1^{(1)}}
\ar"2,4";"3,4"^{\be_1^{(1)}}
\ar"1,2";"2,3"^{\al_2^{(-1)}}
\ar"1,3";"2,4"^{\al_2^{(0)}}
\ar"1,4";"2,5"^{\al_2^{(1)}}
\ar"3,2";"2,3"_{\be_2^{(-1)}}
\ar"3,3";"2,4"_{\be_2^{(0)}}
\ar"3,4";"2,5"_{\be_2^{(1)}}
}}
$$
and
$$
I_{G, W}= \ang{\al_2^{(i)}\al_1^{(i)} - \be_2^{(i)}\be_1^{(i)}, \be_1^{(i+1)}\al_2^{(i)}, \al_1^{(i+1)}\be_2^{(i)},
\al_2^{(i+1)}\al_1^{(i+1)}\al_2^{(i)}, \be_2^{(i+1)}\be_1^{(i+1)}\be_2^{(i)}
\mid i\in \bbZ}.
$$ 
\end{exm}

We have more precise information if we consider the canonical $G$-covering $F\colon A\# G \to A$ stated in Remark \ref{rmk:free-action-can-cov}(2).

\begin{cor}
Let $(Q, I)$ and $W$ be as in Theorem $\ref{thm:quiv-pres-sm-prod}$, and
$F\colon \k(Q, I, W)\# G \to \k(Q, I)$ the canonical $G$-covering.
Then we have a strict commutative diagram
$$
\xymatrix{
\k(Q, I, W)\# G && \k(Q_{G,W}, I_{G,W})\\
 &\k(Q, I)
 \ar"1,1";"1,3"^{\ovl{\ph}}
 \ar"1,1";"2,2"_F
 \ar"1,3";"2,2"^{\k F_{G,W}}
 }
$$
of $\k$-functors.
Therefore we can regard $F_{G,W}$ as a presentation of $F$.
\end{cor}

\begin{proof}
Straightforward.
\end{proof}

\begin{rmk}
By combining with results in \cite{Gr83} we see that
the Galois coverings of a bound quiver $(Q, I)$ with group $G$ constructed in a topological way
explained in the introduction coincide with those having the form
$F_{G,W} \colon (Q_{G,W}, I_{G,W}) \to (Q, I)$ for some $G$-weight $W$.
Therefore the Galois coverings of locally bounded categories
obtained by a topological construction
are covered by the coverings given by
the canonical $G$-coverings of smash products.
\end{rmk}

\section{Brauer graphs}
\label{sec:brauer-graph}

\begin{dfn}
A {\em Brauer permutation} is a quadruple $B:= (E, \si, \ta, m)$ of a set $E$,
permutations $\si, \ta$ of $E$ and a map $m\colon E/\si \to \bbN$
such that $\ta$ is an involution acting freely on $E$
(namely, $\ta^2 = \id_E$ and $\ta e \ne e\ \text{for all } e \in E$), and
each $e \in E$ has a finite $\ang{\si}$-orbit $\ang{\si} e$.
We set $n(e):= |\ang{\si} e|$.
The map $m$ is called the {\em multiplicity} of $B$, and is said to be $trivial$
if it is constant with the value $1$.
\end{dfn}

\begin{rmk}
A triple $(E, \si, \ta)$ of a set $E$ and permutations $\si, \ta$ of $E$ 
such that $\ta$ is an involution acting freely on $E$
is a notion equivalent to a {\em ribbon graph} defined in Adachi--Aihara--Chan \cite{AAC}
under the assumption that $E/\si$ is finite.
We simply call such $(E, \si, \ta)$ a ribbon graph without this assumption
and call $m$ a {\em multiplicity} of the ribbon graph.
Then a Brauer permutation is exactly a ribbon graph with a multiplicity
with the property that each $\ang{\si}$-orbit is finite.
Therefore the notion of Brauer permutation is equivalent to that of Brauer graph defined in \cite{AAC}
with this property, and hence the Brauer graph defined below is the notion equivalent to the usual one
in the case where $E$ is finite.
Note that the set $E$ itself is corresponding to the set of ``half edges'' of a Brauer graph
in the usual sense.

We introduced this notion because (1) it is accurate and simple,
and (2) useful to compute coverings, and (3) it combines
both the corresponding Brauer graph and
the bound quiver of the corresponding Brauer graph algebra as an intermediate one.
\end{rmk}

\begin{exm}
\label{exm:first}
Let $E:= \{1^-, 1^+, 2^-, 2^+\}$ and define a permutation $\si$ of $E$ by
solid arrows in the following diagram ($a \to b$ stands for $\si(a) = b$) and
a permutation $\ta$ of $E$ by
$\ta(i^{\pm}):= i^{\mp}$ ($i = 1, 2$):
$$
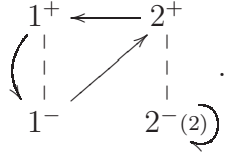
\vcenter{
\xymatrix{
1^+ & 2^+\\
1^- & 2^-\ 
\ar"1,2";"1,1"
\ar"2,1";"1,2"
\ar@/_1pc/"1,1";"2,1"
\ar@{--}"1,1";"2,1"
\ar@{--}"1,2";"2,2"
\ar@(ur,dr)_{(2)}
}}.
$$
Thus $\si = (1^+\ 1^-\ 2^+), \ta = (1^+\ 1^-)(2^+ \ 2^-)$.
Finally define a map $m\colon E/\si \to \bbN$ by the numbers in 
parenthesis inside $\ang{\si}$-orbits of $E$ in the diagram above
(we usually omit the notation $(1)$ standing for the value 1), i.e.,
$m(\ang{\si}1^+) = 1, m(\ang{\si}2^-) = 2$.
Then $(E, \si, \ta, m)$ is a Brauer permutation.
We use this construction throughout the paper in examples.
\end{exm}

\begin{dfn}
\label{dfn:Br-gr-bdquiv}
Let $B = (E, \si, \ta, m)$ be a Brauer permutation.
\begin{enumerate}
\item
The {\em Brauer graph} $\Ga(B)$ defined by $B$ is a triple $(\Ga, \si', m)$, where
\begin{itemize}
\item
$\Ga:= (E/\si, E/\ta, C)$ is a graph with $C\colon E/\ta \to E/\si$ a map defined by
$C(\ang{\ta}e):= \{\ang{\si}e, \ang{\si}\ta e\}$\ $(e \in E)$, and
\item
$\si':= (\si|_V)_{V \in (E/\si)}$ is a sequence (identify $\si'$ with
$\si = \prod_{V \in (E/\si)}(\si|_V)$).
\end{itemize}
\item
The {\em bound Brauer quiver} $(Q(B), I(B))$ defined by $B$
is the following bound quiver $(Q, I)$:
\begin{itemize}
\item
$Q_0:= E/\ta, \quad
Q_1:= \{\al_e\colon\ang{\ta}e\to \ang\ta\si e\mid e\in E\}$,
\item
$
I:= \ang{\al_{\ta\si(e)}\al_e, \mu_e^{m(\ang{\si}e)} - \mu_{\ta e}^{m(\ang{\si}\ta e)} \mid e \in E},
$
\end{itemize}
where $\mu_e:= \al_{\si^{n(e)-1}e} \cdots \al_{\si e}\al_e$ for all $e \in E$.
\item
The {\em Brauer graph algebra $($resp.\ category$)$} $A(B)$ defined by $B$ is the $\k$-algebra
(resp.\ $\k$-category) given by the bound Brauer quiver $(Q, I)$ above.
\end{enumerate} 
\end{dfn}

\begin{exm}
\label{exm:Br-gr-quiv}
Let $B = (E, \si, \ta, m)$ be the Brauer permutation in Example \ref{exm:first}.
Then the Brauer graph $\Ga(B)$ defined by $B$ is presented by
$$
\xymatrix@C=40pt{
\save[]+<-23pt,0pt>*\txt{\tiny 1}\restore
\save[]+<-5pt,10pt>*\txt{\tiny +}\restore
\save[]+<-5pt,-10pt>*\txt{\tiny $-$}\restore
\la \ar@{-}@(dl,ul) & \mu \save[]+<0pt,10pt>*\txt{\tiny (2)}\restore,
\ar@{-}"1,1";"1,2"^2^(0.2)+^(0.8)-
}
$$
where $i:=\ang{\ta} i^+ = \{i^+,i^-\}$ for $(i = 1, 2)$, $\la:= \ang{\si}1^+, \mu:=\ang{\si}2^-$ and $\si'$ is given by the counterclockwise rotation at each vertex, and
$+$ and $-$ stands for ``half edges''.
(This is obtained from $B$ by shrinking the $\ang{\si}$-orbits to vertices
and replacing broken edges by solid ones.)

The bound Brauer quiver $(Q(B), I(B))$ defined by $B$ is equal to the following quiver
$$
\xymatrix{
\save[]+<-27pt,0pt>*\txt{\tiny $\al_{1^+}$}\restore
1 
\ar@(ul,dl) & 2 \ar@(ur,dr)
\save[]+<29pt,0pt>*\txt{\tiny $\al_{2^-}$}\restore
\ar@/^/"1,1";"1,2"^{\al_{1^-}}
\ar@/^/"1,2";"1,1"^{\al_{2^+}}
}
$$
(this is obtained from $B$ by shrinking the broken edges)
with the ideal generated by the relations:
$$
\begin{gathered}
\al_{1^+}^2, \al_{1^-}\al_{2^+}, \al_{2^+}\al_{2^-}, \al_{2^-}\al_{1^-},\\
\al_{2^+}\al_{1^-}\al_{1^+} - \al_{1^+}\al_{2^+}\al_{1^-},
\al_{1^-}\al_{1^+}\al_{2^+} - (\al_{2^-})^{2}.
\end{gathered}
$$
\end{exm}

\subsection{Coverings of Brauer graph algebras by Brauer graph categories}
\begin{dfn}
Let $B = (E, \si, \ta, m)$ be a Brauer permutation and $(Q,I)$ the bound Brauer quiver defined by $B$.
Further let $W\colon E \to G$  be a map, which we call a $G$-{\em weight} on $B$.

\begin{enumerate}
\item
$W$ is said to be {\em homogeneous} if $W$ is a homogeneous weight on $(Q, I)$,
i.e., the following holds for all $e \in E$:
\begin{equation}\label{eq:hmg}
W(\mu_e^{m(\ang{\si}e)}) = W(\mu_{\ta e}^{m(\ang{\si}\ta e)}),
\end{equation}
where we set $W(\al_e):= W(e)$ for all $e \in E$.
\item
$W$ is said to be {\em admissible} if $W(\mu_e^{m(\ang{\si}e)}) = 1$
for all $e \in E$.
Note that if this is the case, then the equality \eqref{eq:hmg} holds with both 
hand sides equal to the unit of $G$.
Thus admissible $G$-weights are homogeneous.
\item
We define permutations $\si_W, \ta_W$ of
$E_W:= E \times G:= \{e_g:= (e, g)\mid e\in E, g\in G\}$ by
$$
\si_W(e_g):= \si(e)_{W(e)g}, \quad \ta_W(e_g):= \ta(e)_g
$$
for all $e \in E, g \in G$.
Then $(E_W, \si_W, \ta_W)$ turns out to be a ribbon graph,
i.e., $\ta_W$ is an involution acting freely on $E_W$
because so is $\ta$ on $E$.
\end{enumerate}

\end{dfn}

\begin{exm}
\label{exm:del-multi}
Let $G:= \ang{a \mid a^2 = 1}$ be the cyclic group of order 2.
Consider the following Brauer permutation $B = (E, \si, \ta, m)$
with a $G$-weight $W$
$$
\vcenter{
\xymatrix{
1^+ & 2^+\\
1^- \ar@(ul,dl) \save[]+<-24pt,0pt>*\txt{\tiny 1}\restore & 2^- \ar@(ur,dr)
\save[]+<24pt,0pt>*\txt{\tiny 1}\restore
\ar@/^10pt/"1,1";"1,2"^{a}
\ar@/^10pt/"1,2";"1,1"^{1}
\ar@{--}"1,1";"2,1"
\ar@{--}"1,2";"2,2"
\ar@{}"1,1";"1,2"|{(2)}
}},
$$
\\
where $W$ is given by writing the values $W(\al)$ at each arrow $\al$. 
Then $W$ is an admissible $G$-weight on $B$, and the ribbon graph
$(E_W, \si_W, \ta_W)$ is given by\\
$$
\vcenter{
\xymatrix{
1_0^+ & 2_0^+ & 1_1^+ & 2_1^+\\
1_0^- \ar@(ul,dl) & 2_0^- \ar@(ur,dr)
&
1_1^- \ar@(ul,dl) & 2_1^- \ar@(ur,dr)
\ar@/^15pt/"1,1";"1,4"
\ar@/^10pt/"1,2";"1,1"
\ar@/^10pt/"1,3";"1,2"
\ar@/^10pt/"1,4";"1,3"
\ar@{--}"1,1";"2,1"
\ar@{--}"1,2";"2,2"
\ar@{--}"1,3";"2,3"
\ar@{--}"1,4";"2,4"
}}.
$$
\end{exm}

Next we consider a multiplicity on the ribbon graph $(E_W, \si_W, \ta_W)$.
We start from the following well known fact.

\begin{lem}
\label{lem:divides}
Let $S$ be a set with an element $s$ and $\si$ a permutation of $S$.
Assume that $T:= \{n \in \bbN \mid \si^n s = s\}$ is not empty
and let $k$ be the minimum element of $T$.
Then $k \mid n$ for all $n \in T$.
\end{lem}

We use this to show the following.

\begin{prp}
Let $B = (E, \si, \ta, m)$ be a Brauer permutation,
and $W$ a $G$-weight on $B$.
Then $W$ is admissible if and only if
$|\ang{\si_W}e_g| < \infty$ and
\begin{equation}
\label{eq:new-multi}
m_W(\ang{\si_W}e_g):= \frac{m(\ang{\si}e)\cdot |\ang{\si}e|}{|\ang{\si_W}e_g|}
\in \bbN
\end{equation}
for all $e \in E, g \in G$.
In particular, if this is the case, $B_W:= (E_W, \si_W, \ta_W, m_W)$
turns out to be a Brauer permutation.
\end{prp}

\begin{proof}
Let $e \in E$, $g \in G$ and set $m:= m(\ang{\si}e)$, $n:= |\ang{\si}e|$ and $k:=|\ang{\si_W}e_g|$.

(\implies).  Assume that $W$ is admissible.
It is enough to show that $k \mid mn$.
By the definition of $n$ and $k$ we have
$\si^n e = e$ and $(\si_W)^k e_g = e_g$.
Since $W$ is admissible, we have
$(\si_W)^{mn}e_g = \si^{mn}(e)_{(W(\si^{n-1}(e))\cdots W(\si^2(e))W(\si(e))W(e))^m g}
= e_g$.
Hence by Lemma \ref{lem:divides} we have $k \mid mn$.

(\impliedby).  Assume that $k < \infty$ and that $t:= mn/k \in \bbN$ (for all $e \in E, g \in G$).
Then
$$
\begin{aligned}
(\si_W)^{mn}e_g &= (\si_W)^{tk}e_g = ((\si_W)^k)^t e_g = e_g, \text{and on the other hand,}\\
(\si_W)^{mn}e_g &= ((\si_W)^n)^m e_g = \si^{mn}(e)_{(W(\si^{n-1}(e))\cdots W(\si^2(e))W(\si(e))W(e))^m g}\\
&= e_{(W(\si^{n-1}(e))\cdots W(\si^2(e))W(\si(e))W(e))^m g}.
\end{aligned}
$$
Hence we have $g = (W(\si^{n-1}(e))\cdots W(\si^2(e))W(\si(e))W(e))^m g$ and thus
$$
W(\mu_e^{m})  = (W(\si^{n-1}(e))\cdots W(\si^2(e))W(\si(e))W(e))^m = 1.
$$
This holds for all $e \in E$, and hence $W$ is admissible.
\end{proof}

\begin{cor}
\label{cor:trivial-multi}
If $B$ is a Brauer permutation with a trivial multiplicity and $W$ is a $G$-weight on $B$,
then $B_W$ has a trivial multiplicity.
\end{cor}

\begin{proof}
Assume that $B$ has a trivial multiplicity.
Let $e \in E, g \in G$.
It is enough to show that $m_W(\ang{\si_W}e_g) = 1$.
In general, we have $|\ang{\si}e| \le |\ang{\si_W}e_g|$
because the first projection yields a surjection $\ang{\si_W}e_g \to \ang{\si}e$.
Therefore by \eqref{eq:new-multi} we have $m_W(\ang{\si_W}e_g) \le m(\ang{\si}e) = 1$.
Here $m_W(\ang{\si_W}e_g)$ is a natural number because $W$ is admissible.
Hence $m_W(\ang{\si_W}e_g) = 1$.
\end{proof}

We are now in a position to give a way to compute coverings of Brauer graph algebras
by Brauer graph categories.

\begin{thm}
\label{thm:cover-Brauer}
Let $B = (E, \si, \ta, m)$ be a Brauer permutation,
$W$ an admissible $G$-weight on $B$, and $(Q, I)$ the bound Brauer quiver
defined by $B$.
Then $(Q, I, W)$ gives a $G$-graded category $\k(Q, I, W)$ and
the smash product $\k(Q, I, W) \# G$ is given by the bound Brauer quiver
defined by the Brauer permutation $B_W$.
\end{thm}

\begin{proof}
Let $(Q_W, I_W)$ be the bound Brauer quiver of the Brauer permutation $B_W$.
By Theorem \ref{thm:quiv-pres-sm-prod} we have $\k(Q, I, W) \# G \iso \k(Q_{G,W},I_{G,W})$.
Therefore it is enough to construct a quiver isomorphism $F\colon Q_{G,W} \to Q_W$ such that
$\k F(I_{G,W}) = I_W$.
By definitions we have
$$
\begin{aligned}
(Q_{G,W})_0 &= Q_0 \times G = E/\ta \times G = \{(\ang{\ta}e)^{(g)} = (\{e, \ta e\}, g) \mid e \in E, g \in G\}, \text{and}\\
(Q_W)_0 &= E_W/\ta_W = (E \times G)/\ta_W = \{\ang{\ta_W}e_g=\{(e, g), (\ta e, g)\} \mid e \in E, g\in G\}.
\end{aligned}
$$
Moreover, 
$$
\begin{aligned}
(Q_{G,W})_1 &= \{ \al^{(g)} \colon x^{(g)} \to y^{(W(\al)g)} \mid \al\colon x \to y \text{ in } Q_1, g \in G\}\\
&= \{\al_e^{(g)}\colon (\ang{\ta}e)^{(g)} \to (\ang{\ta}\si e)^{(W(e)g)} \mid  e \in E, g \in G\}
\end{aligned}
$$
and
$$
\begin{aligned}
(Q_W)_1&= \{\al_{e_g} \colon \ang{\ta_W}e_g \to \ang{\ta_W}\si_We_g \mid e \in E, g \in G\}\\
&= \{\al_{e_g} \colon \ang{\ta_W}e_g \to \ang{\ta_W}\si(e)_{W(e)g} \mid e \in E, g \in G\}.
\end{aligned}
$$
Therefore the correspondence
$$
(\ang{\ta}e)^{(g)} \mapsto \ang{\ta_W}e_g,\quad \al_e^{(g)} \mapsto \al_{e_g}
$$
for all $e \in E, g \in G$ defines a quiver isomorphism $F\colon Q_{G,W} \to Q_W$.

Now we have
$$
I_{G,W} = \ang{(\al_{\ta\si(e)}\al_e)^{(g)},
((\mu_e)^{m(\ang{\si}e)})^{(g)} - ((\mu_{\ta e})^{m(\ang{\si}\ta e)})^{(g)} \mid e \in E, g \in G}
$$
and
$$
I_W =\ang{\al_{\ta_W\si_W(e_g)}\al_{e_g},
\mu_{e_g}^{m_W(\ang{\si_W}e_g)} - \mu_{\ta_W e_g}^{m_W(\ang{\si_W}\ta_W e_g)}\mid e \in E, g \in G}.
$$
Therefore it is enough to show the following equalities.
\begin{align}
F((\al_{\ta\si(e)}\al_e)^{(g)}) &= \al_{\ta_W\si_W(e_g)}\al_{e_g},\label{eq:1}\\
F(((\mu_e)^{m(\ang{\si}e)})^{(g)}) &= \mu_{e_g}^{m_W(\ang{\si_W}e_g)}, \text{and}\label{eq:2}\\
F(((\mu_{\ta e})^{m(\ang{\si}\ta e)})^{(g)}) &= \mu_{\ta_W e_g}^{m_W(\ang{\si_W}\ta_W e_g)}\label{eq:3}.
\end{align}
The equality \eqref{eq:1} follows by 
$$
\text{LHS}  = F((\al_{\ta\si(e)})^{(W(e)g)}\al_e^{(g)}) = \al_{\ta\si(e)_{W(e)g}}\al_{e_g}
= \text{RHS}.
$$
To show \eqref{eq:2} we set $m:= m(\ang{\si}e), n:= n(e), t:= m_W(\ang{\si_W}e_g), k:= |\ang{\si_W}e_g|$.
Then $tk = mn$ and we have
\begin{equation}\label{eq:4}
((\mu_e)^{m})^{(g)}= \al_{\si^{mn-1}(e)}^{(W(\si^{mn-2}(e))\cdots W(\si(e))W(e))}\cdots\al_{\si(e)}^{(W(e)g)}\al_{e_g}
\end{equation}
and
\begin{equation}\label{eq:5}
\mu_{e_g}^t =(\al_{\si_W^{k-1}(e_g)}\cdots \al_{\si_W(e_g)}\al_{e_g})^t
= (\al_{\si^{k-1}(e)_{W(\si^{k-2}(e))\cdots W(\si(e)) W(e)g}}\cdots \al_{\si(e)_{W(e)g}} \al_{e_g})^t.
\end{equation}
Since $\si_W^k(e_g) = e_g$, we have $\si^k(e) = e$ and $W(\si^{k-1}(e))\cdots W(\si(e))W(e) = 1$.
Therefore \eqref{eq:4} and \eqref{eq:5} shows the equality \eqref{eq:2}.
Finally, since 
$\mu_{\ta_W(e_g)}^{m_W(\ang{\si_W}\ta_W(e_g))} = \mu_{\ta(e)_g}^{m_W(\ang{\si_W}\ta(e)_g)}$,
the equality \eqref{eq:3} follows from \eqref{eq:2} by substituting $\ta(e)$ for $e$.
\end{proof}

We next apply Theorem \ref{thm:cover-Brauer} to reduce Brauer graphs to those with simpler structures.
See Remark \ref{rmk:GSS-der-eq}(1) for relationships between our propositions below and
those by Green--Schroll--Snashall \cite{GSS14}.

\subsection{Deletion of multiplicity}
As was shown in Example \ref{exm:del-multi} we can make the multiplicity trivial by forming
smash products, which we call a {\em deletion of multiplicity}.
The following makes it possible to delete multiplicity for each case at once (not step by step).

\begin{prp}
\label{prp:del-multi}
Let $B = (E, \si, \ta, m)$ be a Brauer permutation, and $\{e_1, e_2, \dots, e_t\}$
a complete set of representatives of the set $\{\ang{\si}e \in E/\si \mid m(\ang{\si}e) > 1\}$
$($of vertices of $\Ga(B)$ at which the value of multiplicity is $> 1)$.
Set $m_i:= m(\ang{\si}e_i)$ for all $i=1, 2,\dots, t$, and $m$ to be the least common multiple of all $m_i$'s.
Take $G:= \ang{a \mid a^m =1}$ to be the cyclic group of order $m$, and define a $G$-weight
$W$ by
$$
W(e):=
\begin{cases}
b_i:= a^{m/m_i} &\text{if $e = e_i$ for some $i=1, 2,\dots, t$},\\
1 &\text{otherwise}
\end{cases}
$$
for all $e \in E$.
Then $W$ is admissible, and the Brauer permutation $B_W$ has a trivial multiplicity,
namely $m_W(\ang{\si_W}e_g) = 1$ for all $e_g \in E_W$ $(e \in E, g \in G)$.
\end{prp}
 
\begin{proof}
Let $e \in E, g \in G$.
Then $m_W(\ang{\si_W}e_g)$ is computed by the formula \eqref{eq:new-multi}.
Note first that the order of $b_i$ is equal to $m_i$ for all $i$, from which
it is obvious that $W$ is admissible.

\subsection*{Case 1} $\ang{\si}e \ne \ang{\si}e_i$ for all $i=1, 2,\dots, t$.
Set $n:= n(e)$.

In this case we have $e_i \not\in \ang{\si}e$ for all $i$ and $m(\ang{\si}e) = 1$.
The former shows that
$$
\ang{\si_W}e_g = \{e_g, \si(e)_g, \dots, \si^{n-1}(e)_g \},
$$
and thus $|\ang{\si_W}e_g| = |\ang{\si}e|$,
which together with the latter shows that $m_W(\ang{\si_W}e_g) = 1$
by \eqref{eq:new-multi}.

\subsection*{Case 2}

$\ang{\si}e = \ang{\si}e_i$ for some $i=1, 2,\dots, t$,
say $e = \si^s(e_i)$ for some $s = 0, 1,\dots, n_i-1$,
where we set $n_i:= n(e_i)$.  Set also $m_i:= m(\ang{\si}e_i)$.

A direct calculation gives us
$$
\ang{\si_W}e_g = \{\si^j(e_i)_{b_i^k g} \mid j \in \{0, 1, \dots, n_i -1\}, k \in \{0, 1,\dots, m_i -1\} \},
$$
which shows that
$|\ang{\si_W}e_g| = m_i n_i$, and hence by \eqref{eq:new-multi} we have 
$$
m_W(\ang{\si_W}e_g) = \frac{m_i n_i}{|\ang{\si_W}e_g|} = 1.
$$
\end{proof}

\subsection{Deletion of loops in Brauer graphs}
Next we delete all the loops in Brauer graphs by forming double coverings
(i.e., smash products with the cyclic group of order 2).

\begin{prp}
\label{prp:del-loop}
Let $B = (E, \si, \ta, m)$ be a Brauer permutation, and $\{e_1, e_2, \dots, e_t\}$
a complete set of representatives of the set $\{\ang{\ta}e \in E/\ta \mid \ang{\si}e = \ang{\si}\ta e\}$
$($of  loops in $\Ga(B))$.
Take $G:= \ang{a \mid a^2=1}$ to be the cyclic group of order $2$,
and define a $G$-weight $W$ on $B$ by
$$
W(e):=
\begin{cases}
a &\text{ if $e \in \{e_i, \ta(e_i) \mid i \in \{1, 2, \dots, t\}\}$,}\\
1 &\text{otherwise}
\end{cases}
$$
for all $e \in E$.
Then $W$ is admissible, and the Brauer graph $\Ga(B_W)$ has no loops.
\end{prp}

\begin{proof}
By definition it is obvious that $W$ is admissible.
Assume that $\Ga(B_W)$ has a loop $\ang{\ta_W}e_g$
$(e \in E, g \in G)$.
Then $\ang{\si_W}e_g = \ang{\si_W}\ta_W(e_g)$, and
$\ta_W(e_g) = \si_W^i(e_g)$
for some unique $i \in \{0, 1,\dots, |\ang{\si_W}e_g|-1\}$.
Therefore
$$
\ta_W(e_g) =\ta(e)_g \text{ and } \si_W^i(e_g)  = \si^i(e)_{W(\si^{i-1}(e))\cdots W(\si(e))W(e)g}
$$
show that
\begin{equation}
\label{eq:loop-contradiction}
\ta(e) = \si^i(e) \text{ and }
g = W(\si^{i-1}(e))\cdots W(\si(e))W(e)g.
\end{equation}
The latter in \eqref{eq:loop-contradiction} shows that
$W(\si^{i-1}(e))\cdots W(\si(e))W(e) = 1$.
However, the former shows that 
$\ang{\si}\ta(e) = \ang{\si}e$,
which means that $\ang{\ta}e$ is a loop in $\Ga(B)$, i.e., $e = e_j$ or $e=\ta(e_j)$
for some $j \in \{1,2,\dots, t\}$.
In particular, $W(e) = a$.
Again by the former in \eqref{eq:loop-contradiction} we see that
$W(\si^{i-1}(e)) =1, \cdots, W(\si(e)) = 1$.
Then $W(e) = a$ shows that $W(\si^{i-1}(e))\cdots W(\si(e))W(e) = a$, a contradiction.
As a consequence, $\Ga(B_W)$ has no loops.
\end{proof}

\subsection{Deletion of multiple edges in Brauer graphs}
Next we delete all multiple edges from a Brauer graph by the smash product with
a group that is the product of cyclic groups.

\begin{prp}
\label{prp:del-double}
Let $B = (E, \si, \ta, m)$ be a Brauer permutation without loops in $\Ga(B)$, and
$\{e_1, e_2,\dots, e_t\}$ a complete set of representatives of the set
$$
\{\ang{\si}e \in E/\si \mid \ang{\si}e = \ang{\si}f, \ang{\si}\ta(e) = \ang{\si}\ta(f)
\text{ for some $e \ne f \in E$}\}
$$
of vertices of $\Ga(B)$ such that there exist multiple edges
between the vertices $\ang{\si}e$ and $\ang{\si}\ta(e)$,
namely that there exists an edge $\ang{\ta}f$ between them different from the edge $\ang{\ta}e$.
Set $n_i:= n(e_i)$, let $G_i:= \ang{a_i \mid a_i^{n_i} = 1}$ be the cyclic group of order $n_i$
for all $i \in \{1, 2,\dots, t\}$, and take $G:= \prod_{i=1}^t G_i$, where we regard each $G_i$
to be a subgroup of $G$ by the canonical injection $G_i \to G$.
Define a $G$-weight $W$ on $B$ by
$$
W(e):=
\begin{cases}
a_i &\text{if $e \in \ang{\si}e_i$ for some $i$},\\
1 &\text{otherwise}
\end{cases}
$$
for all $e \in E$.
Then $W$ is admissible, and the Brauer graph $\Ga(B_W)$ has neither multiple edges nor loops.
\end{prp}

\begin{proof}
It is obvious that $W$ is admissible by construction.
Assume that $\Ga(B_W)$ has multiple edges.
Then there exist two distinct elements $e_g, f_h \in E_W$ $(e, f \in E, g, h \in G)$ such that
$$
\ang{\si_W}e_g = \ang{\si_W}f_h \text{ and } \ang{\si_W}\ta_W(e_g) = \ang{\si_W}\ta_W(f_h).
$$
Set $n, n'$ to be the cardinality of the former set and the latter set, respectively.
Then there exist some integers $0 \le i \le n-1$ and $0 \le j \le n'-1$ such that
$$
f_h = \si_W^i(e_g) \text{ and } \ta_W(f_h) = \si_W^j\ta_W(e_g),
$$
which imply the following equalities.
\begin{gather}
f = \si^i(e), \ta(f) = \si^j\ta(e) \label{eq:E-part}\\
h = W(\si^{i-1}(e))\cdots W(\si(e))W(e)g, h = W(\si^{i-1}\ta(e))\cdots W(\si\ta(e))W(\ta(e))g
\label{eq:G-part}
\end{gather}
By \eqref{eq:E-part} we have $e = e_k, \ta(e) = e_l$ for some $k, l \in \{1, 2, \dots, t\}$, and
$$
\begin{gathered}
W(e) = W(\si(e)) = \cdots = W(\si^{i-1}(e)) = a_k, n = n_k, \text{and}\\
W(\ta(e)) = W(\si\ta(e)) = \cdots = W(\si^{i-1}\ta(e)) = a_l, n' = n_l.
\end{gathered}
$$
Then \eqref{eq:G-part} shows that
$h = a_k^i g, h = a_l^j g$, and hence we have
$G_k \ni a_k^i = a_l^j \in G_l$.
But since $G_k \cap G_l = \{1\}$, we have $a_k^i = a_l^j = 1$, which implies that $h = g$ and
that $n_k \mid i$.  The latter gives $i = 0$, and $f = e$ because $0 \le i \le n_k - 1$.
Hence $e_g = f_h$, a contradiction.
As a consequence, $\Ga(B_W)$ has no multiple edges.

Assume that $\Ga(B)$ has no loops but
that $\Ga(B_W)$ has a loop $\ang{\ta_W}e_g$ $(e \in E, g \in G)$.
Then $\ang{\si_W}e_g = \ang{\si_W}\ta_W(e_g) =\ang{\si_W}\ta(e)_g$, which implies that
$\ta(e) \in \ang{\si}e$, i.e., $\ang{\ta}e$ is a loop in $\Ga(B)$, a contradiction.

The last part of the proof of Proposition \ref{prp:del-loop} works for the remaining part:
If $B$ has a trivial multiplicity, then so does $B_W$ because $W$ is admissible.
\end{proof}

Sometimes it is possible to delete multiple edges by using simpler group.
We record one of those cases using one cyclic group in the following.

\begin{prp}
\label{prp:del-double-one}
Let $B = (E, \si, \ta, m)$ be a Brauer permutation without loops in $\Ga(B)$, and
$$
V:= \{\ang{\si}e \in E/\si \mid \ang{\si}e = \ang{\si}f, \ang{\si}\ta(e) = \ang{\si}\ta(f)
\text{ for some $e \ne f \in E$}\}
$$
the set of vertices of $\Ga(B)$ such that there exist multiple edges
between the vertices $\ang{\si}e$ and $\ang{\si}\ta(e)$.
Consider the graph $\De$ obtained from the subgraph of\/ $\Ga(B)$ consisting of the vertices in $V$ and 
multiple edges between them by replacing the set of all multiple edges between $\ang{\si}e$ and $\ang{\si}f$
to a single edge between them for all $\ang{\si}e, \ang{\si}f \in V$ $(e, f \in E)$.

Assume that $\De$ is a tree.
Then there exists a coloring of vertices $V$ of $\De$ by two colors $\{c, c'\}$
such that for each edge of $\De$ the colors of two end vertices are different.
Let $V'$ be the set of vertices of $\De$ with the color $c'$, and
$\{e_1, e_2,\dots, e_t\}$ a complete set of representatives of the set $V'$.
Set $n_i:= n(e_i)$ for all $i$, and $n$ to be the least common multiple of $n_i$'s.
Take $G:= \ang{a \mid a^n = 1}$ to be the cyclic group of order $n$.
Define a $G$-weight $W$ on $B$ by
$$
W(e):=
\begin{cases}
b_i:= a^{n/n_i} &\text{if $e \in \ang{\si}e_i$ for some $i$},\\
1 &\text{otherwise}
\end{cases}
$$
for all $e \in E$.
Then $W$ is admissible, and the Brauer graph $\Ga(B_W)$ has neither multiple edges nor loops.
\end{prp}

\begin{proof}
Note that the order of $b_i$ is $n_i$ for all $i$, which shows that $W$ is admissible.
Assume that $\Ga(B_W)$ has multiple edges.
Then as was described in the previous proof,
there exist two distinct elements $e_g, f_h \in E_W$ $(e, f \in E, g, h \in G)$ such that
$$
\ang{\si_W}e_g = \ang{\si_W}f_h \text{ and } \ang{\si_W}\ta_W(e_g) = \ang{\si_W}\ta_W(f_h).
$$
Set $n, n'$ to be the cardinality of the former set and the latter set, respectively.
Then there exist some integers $0 \le i \le n-1$ and $0 \le j \le n'-1$ such that
$$
f_h = \si_W^i(e_g) \text{ and } \ta_W(f_h) = \si_W^j\ta_W(e_g),
$$
which imply the equalities \eqref{eq:E-part}, \eqref{eq:G-part}.
By construction and by \eqref{eq:E-part} we may assume that $e = e_k$ for some $k\in \{1, 2, \dots, t\}$, and
 that $\ang{\si}\ta(e) \in V \setminus V'$.
 Then we have
 $$
 \begin{gathered}
W(e) = W(\si(e)) = \cdots = W(\si^{i-1}(e)) = b_k, n = n_k, \text{and}\\
W(\ta(e)) = W(\si\ta(e)) = \cdots = W(\si^{i-1}\ta(e)) = 1.
\end{gathered}
$$
Then \eqref{eq:G-part} shows that
$h = b_k^i g$ and $h =  g$, and hence $b_k^i = 1$.
Therefore we have $n_k \mid i$, and hence $i = 0$ because $0 \le i \le n_k - 1$.
Thus $f = e$ .
Hence $e_g = f_h$, a contradiction.
As a consequence, $\Ga(B_W)$ has no multiple edges.
The rest is proved by the same argument as in the proof of Proposition \ref{prp:del-double}. 
\end{proof}

\begin{rmk}
\label{rmk:GSS-der-eq}
 (1) In \cite[Propositions 6.4, 6.5, and 6.6]{GSS14} Green--Schroll--Snashall gave statements
similar to Propositions \ref{prp:del-multi}, \ref{prp:del-loop} and \ref{prp:del-double}.
In particular, Propositions \ref{prp:del-multi} and \ref{prp:del-double}
are essentially the same as their Propositions 6.4 and 6.6, respectively,
but we gave here simple and complete proofs.
On the other hand, Proposition \ref{prp:del-loop} is stronger than their Proposition 6.5:
They used a bigger group that depends on Brauer graphs, while
we used a very small group, the cyclic group of order 2 and the used group
is always the same.
We also added a simpler case of deletion of multiple edges in Proposition \ref{prp:del-double-one}.

(2) Using Propositions \ref{prp:del-multi}, \ref{prp:del-loop} and \ref{prp:del-double}
we can delete multiplicity, loops and multiple edges from Brauer graphs
by forming finite coverings (smash products with finite groups).
In particular, deletion of loops is an easy procedure, and uniformly we can take the group $G$
as a cyclic group of order 2.
Therefore, for instance, the derived equivalence classification of Brauer graph algebras
might be reduced to that of Brauer graph algebras without loops
using a covering theory for derived equivalences developed in
\cite{Asa97, Asa99, Asa02, Asa11, Asa15, Asa-a, Asa-b}.
\end{rmk}

\subsection{Deletion of cycles}
Finally we delete all cycles in a Brauer graph to have a tree by using an infinite group.
Since loops and double edges are special types of cycles,
we can use this procedure also to delete those at the same time.
This can be done independently of the multiplicity.

\begin{prp}\label{prp:del-cycle}
Let $B = (E, \si, \ta, m)$ be a Brauer permutation with a cycle in its Brauer graph.
Let $\{e_1, e_2, \dots, e_t\}$ be a complete set of representatives of the set $V$ of all
$\ang{\si}$-orbits $\ang{\si}e$ in $E/\si = \Ga(B)_0$ such that there exists a cycle in $\Ga(B)$
through the vertex $\ang{\si}e$.
For each $i = 1,2,\dots, t$ we set $n_i:= n(e_i)$ and $G_i:= \ang{a_i}$ to be an infinite
cyclic group if $n_i \ge 2$, else to be the unit group $\{1\}$.
Take $G:= \prod_{i=1}^t G_i$, where we regard each $G_i$ a subgroup of $G$
by the canonical injections $G_i \to G$.
Define a $G$-weight $W$ by
$$
W(e):=
\begin{cases}
a_i &\text{if $e = \si^j(e_i)$ for some $1 \le i \le t, 0 \le j \le n_i -2$ with $n_i \ge 2$},\\
a_i^{-n_i+1} & \text{if $e = \si^{n_i-1}(e_i)$ for some $1 \le i \le t$ with $n_i \ge 2$},\\
1 & \text{otherwise}
\end{cases}
$$
for all $e \in E$.
Then $W$ is admissible, and $B_W$ has a trivial multiplicity with $\Ga(B_W)$
an infinite tree.
\end{prp}

\begin{proof}
Note first that there exists some $i = 1,2,\dots, t$ such that $n_i \ge 2$.
Indeed, if this does not hold and if $\Ga(B)$ has a cycle
\begin{equation}
\label{eq:cycle}
C: \xymatrix{
\ang{\si}f_{0} & \ang{\si}f_{1} & \cdots & \ang{\si}f_{r-1} &
\ang{\si}f_{r}\ \ (f_{r} = f_{0})
\ar@{-}"1,1";"1,2"
\ar@{-}"1,2";"1,3"
\ar@{-}"1,3";"1,4"
\ar@{-}"1,4";"1,5"
}
\end{equation}
for some $f_0, f_1, \dots, f_{r-1} \in E$ with $r$ minimal among all cycles, then
$\{f_0, f_1, \dots, f_{r-1}\} \subseteq \{e_i, e_2, \dots, e_t\}$ and
$\ang{\si}e_i = \{e_i\}$ for all $i$, thus we must have
$\ta(f_0) = f_1, \ta(f_1) = f_2, \dots, \ta(f_{r-1}) = f_0$.
But since $\ta$ is an involution acting freely on $E$,
we have $f_0 = \ta(f_1) = f_2$ and $f_0 \ne f_1$.
The minimality of $r$ shows that $r = 2$ and we must have double edges
between $\ang{\si}f_{0} = \{f_0\}$ and $\ang{\si}f_{1} = \{f_1\}$,
which is not possible.
In particular, $G$ is an infinite group by construction.

Note next that $W(\mu_e) = 1$ for all $e \in E$ by construction, which shows that
$W$ is admissible.
More precisely, for each $i \in \{1,2,\dots, t\}$, $j \in \{0, 1, \dots, n_i-1\}$ and
$k \in \{1,2,\dots, n_i\}$ we have
\begin{equation}
\label{eq:1-cycle}
W(\si^{j+k-1}(e_i))\cdots W(\si^{j+1}(e_i))W(\si^j(e_i)) = 1
\text{ if and only if }k = n_i.
\end{equation} 
 In particular, this shows that
 $|\ang{\si_W}e_g| = |\ang{\si}e|$ for all $e_g \in E_W$.
 Indeed, this is trivial if $\ang{\si}e \not\in V$.
 Otherwise $\ang{\si}e = \ang{\si}e_i$ for some $i \in \{1,2,\dots, t\}$, and
 $(\le)$ follows from \eqref{eq:1-cycle} for $j=0, k=n_i$
 and $(\ge)$ holds in general as in the proof of Corollary \ref{cor:trivial-multi}.

Finally assume that $\Ga(B_W)$ is not a tree.
Then we have a cycle $\tilde{C}$ of the form
$$
\xymatrix{
\ang{\si_W}f_{0,g_0} & \ang{\si_W}f_{1,g_1} & \cdots & \ang{\si_W}f_{r-1,g_{r-1}} &
\ang{\si_W}f_{r,g_r}\ \ (f_{r,g_r} = f_{0, g_0})
\ar@{-}"1,1";"1,2"
\ar@{-}"1,2";"1,3"
\ar@{-}"1,3";"1,4"
\ar@{-}"1,4";"1,5"
}
$$
in $\Ga(B_W)$ for some $f_{i,g_i} \in E_W\ (f_i \in E, g_i \in G, i=1,2, \dots, r)$ with $r \ge 1$.
We may assume that $r$ is minimal among such numbers.
Then for each  $i = 0, 1, \dots, r-1$ we have
\begin{equation}
\label{eq:induction-vertex}
f_{i+1, g_{i+1}} = \ta_W\si_W^{k_{i}}(f_{i,g_{i}})
\end{equation}
for some integer $k_i$ with $1 \le k_i \le |\ang{\si_W}f_{i,g_i}| = |\ang{\si}f_i| = n_{u(i)}$.
By applying the first and second projections to \eqref{eq:induction-vertex} we have
\begin{align}\label{eq:pr1}
f_{i+1} &= \ta\si^{k_{i}}(f_{i}) \text{ and}\\
\label{eq:pr2}
g_{i+1} &= W(\si^{k_{i}-1}(f_{i}))\cdots W(\si(f_{i}))W(f_{i})g_{i}.
\end{align}
By \eqref{eq:pr1} there exists a cycle $C$ of the form \eqref{eq:cycle}
in $\Ga(B)$ through the vertex $\ang{\si}f_0$.
Thus all $\ang{\si}f_i$ are in $V$, and hence there exists a map $u\colon \{1,2,\dots, r\} \to \{1,2,\dots, t\}$ such that for each $i \in \{1,2,\dots, r\}$
we have $\ang{\si}f_i = \ang{\si}e_{u(i)}$, thus
there exists
some $j_i \in \{1,2,\dots, n_{u(i)}\}$ such that $f_i = \si^{j_i}e_{u(i)}$.
Then by \eqref{eq:pr2} we have
$$
g_{i+1} = W(\si^{j_i +k_{i}-1}(e_{u(i)}))\cdots W(\si^{j_i+1}(e_{u(i)}))W(\si^{j_i}(e_{u(i)}))g_i,
$$
where 
\begin{equation}\label{eq:form-W}
W(\si^{j_i +k_{i}-1}(e_{u(i)}))\cdots W(\si^{j_i+1}(e_{u(i)}))W(\si^{j_i}(e_{u(i)})) = a_{u(i)}^{l_i}
\end{equation}
for some integer $l_i \ge 0$
by construction of $W$.
Therefore we have
\begin{equation}
g_{i+1} = a_{u(i)}^{l_i}g_i.
\end{equation}
Set $e_g:= f_{0,g_0}$.
Then this yields
$$
g_{i+1} = a_{u(i)}^{l_i}\cdots a_{u(1)}^{l_1}a_{u(0)}^{l_0}g
$$
for all $i = 0, 1, \dots, r-1$,
In particular, for $i = r-1$ we have
$$
g = a_{u(r-1)}^{l_{r-1}}\cdots a_{u(1)}^{l_1}a_{u(0)}^{l_0}g.
\text{ and hence }
1 = a_{u(r-1)}^{l_{r-1}}\cdots a_{u(1)}^{l_1}a_{u(0)}^{l_0}.
$$
The minimality of $r$ implies that the map $u$ is injective, thus
$u(0), u(1), \dots, u(r-1)$ are pairwise different.
Therefore by definition of $G$ we have 
$a_{u(i)}^{l_i} = 1$ for all $i = 0, 1, \dots, r-1$.
This together with \eqref{eq:form-W} and \eqref{eq:1-cycle} shows that
for each $i = 0, 1, \dots, r-1$ we have $k_i = n_i$, $g_i = g$,
and hence $f_{i+1,g_{i+1}} = \ta(f_i)_g =\ta_W(f_{i,g_i})$.
Then since $\ta_W$ is an involution acting freely on $E_W$,
the same argument as in the beginning of  the proof applies to
have $r = 2$ and $\tilde{C}$ cannot be a cycle in $\Ga(B_W)$,
which is a contradiction.
\end{proof}

\begin{rmk}
In many cases we can take the group $G$ much smaller than
in Proposition \ref{prp:del-cycle} as shown in Example \ref{exm:del-cycle}.
\end{rmk}

\section{Examples}

In this section we collect some examples of smash products of Brauer graph algebras
to illustrate the contents of the previous section.
\begin{exm}[Deletion of multiplicity]
Consider the following Brauer permutation $B$ with a non-trivial multiplicity
and its Brauer graph $\Ga(B)$:
$$
\xymatrix{
1^+
\ar@(ul,dl)^{(2)} & {1^-\ } \ar@(ur,dr)_{(3)}
\ar@{--}"1,1";"1,2"
}
\qquad\qquad
\xymatrix@C=40pt
{
\la & \mu.
\ar@{-}"1,1";"1,2"^1^(0){(2)}^(1){(3)}
}
$$
To apply Proposition \ref{prp:del-multi} we take $G:= \ang{a \mid a^6=1}$ and define
an admissible $G$-weight $W$ as follows:
$$
\xymatrix{
1^+
\ar@(ul,dl)^{(2)}_{a^3} & {1^-\ } \ar@(ur,dr)_{(3)}^{a^2.}
\ar@{--}"1,1";"1,2"
}
$$
Then $B_W$ and $\Ga(B_W)$ are as follows:
$$
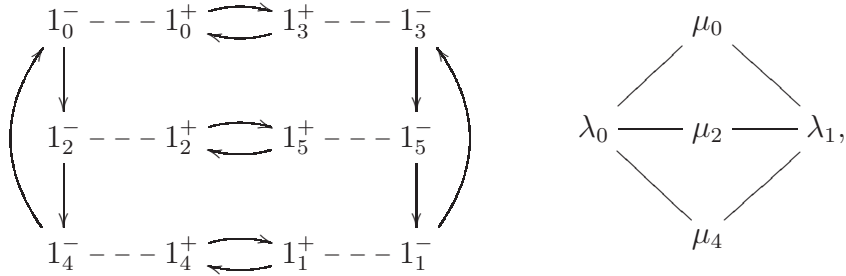
\xymatrix{
1^-_0 & 1^+_0 & 1^+_3 & 1^-_3\\
1^-_2 & 1^+_2 & 1^+_5 & 1^-_5\\
1^-_4 & 1^+_4 & 1^+_1 & 1^-_1
\ar@{--}"1,1";"1,2"
\ar@{--}"1,3";"1,4"
\ar@{--}"2,1";"2,2"
\ar@{--}"2,3";"2,4"
\ar@{--}"3,1";"3,2"
\ar@{--}"3,3";"3,4"
\ar@/^/"1,2";"1,3"
\ar@/^/"1,3";"1,2"
\ar@/^/"2,2";"2,3"
\ar@/^/"2,3";"2,2"
\ar@/^/"3,2";"3,3"
\ar@/^/"3,3";"3,2"
\ar"1,1";"2,1"
\ar"2,1";"3,1"
\ar@/^20pt/"3,1";"1,1"
\ar"1,4";"2,4"
\ar"2,4";"3,4"
\ar@/_20pt/"3,4";"1,4"
}
\qquad\qquad
\xymatrix{
&\mu_0\\
\la_0 & \mu_2 & \la_1,\\
&\mu_4
\ar@{-}"2,1";"1,2"
\ar@{-}"2,1";"2,2"
\ar@{-}"2,1";"3,2"
\ar@{-}"2,3";"1,2"
\ar@{-}"2,3";"2,2"
\ar@{-}"2,3";"3,2"
}
$$
where vertices $1^{\pm}_{a^i}$ are denoted by $1^{\pm}_{i}$ for short.
Certainly $\Ga(B_W)$ has a trivial multiplicity.
\end{exm}

\begin{exm}[Deletion of loops]
Let $B$ be the following Brauer permutation with two loops in its Brauer graph $\Ga(B)$:
$$
\vcenter{
\xymatrix{
1^+ & 2^+ & 3^+\\
1^- & 2^- & 3^-
\ar"1,2";"1,1"
\ar"2,1";"1,2"
\ar@/_10pt/"1,1";"2,1"
\ar"1,3";"2,2"
\ar"2,2";"2,3"
\ar@/_10pt/"2,3";"1,3"
\ar@{--}"1,1";"2,1"
\ar@{--}"1,2";"2,2"
\ar@{--}"1,3";"2,3"
}}
\qquad\qquad
\xymatrix{
\la
\ar@{-}@(ul,dl)_1 & {\mu} \ar@{-}@(ur,dr)^3
\ar@{-}"1,1";"1,2"^2
}.
$$
To Apply Proposition \ref{prp:del-loop} we take $G:=\ang{a \mid a^2 = 1}$
and define an admissible $G$-weight $W$ by
$$
\vcenter{
\xymatrix{
1^+ & 2^+ & 3^+\\
1^- & 2^- & 3^-
\ar"1,2";"1,1"_1
\ar"2,1";"1,2"_a
\ar@/_10pt/"1,1";"2,1"_a
\ar"1,3";"2,2"_a
\ar"2,2";"2,3"_1
\ar@/_10pt/"2,3";"1,3"_a
\ar@{--}"1,1";"2,1"
\ar@{--}"1,2";"2,2"
\ar@{--}"1,3";"2,3"
}}.
$$
Then $B_W$ and $\Ga(B_W)$ are computed as follows:
$$
\vcenter{
\xymatrix{
1^+_1 & 1^-_a & 2^+_1 & a^-_1 & 3^-_1 & 3^+_a\\
1^-_1 & 1^+_a & 2^+_a & 2^-_a & 3^+_1 & 3^-_a
\ar"1,1";"1,2"
\ar"1,2";"1,3"
\ar@/_15pt/"1,3";"1,1"
\ar"1,4";"1,5"
\ar"1,5";"1,6"
\ar@/_15pt/"1,6";"1,4"
\ar"2,3";"2,2"
\ar"2,2";"2,1"
\ar@/_15pt/"2,1";"2,3"
\ar"2,6";"2,5"
\ar"2,5";"2,4"
\ar@/_15pt/"2,4";"2,6"
\ar@{--}"1,1";"2,1"
\ar@{--}"1,2";"2,2"
\ar@{--}"1,5";"2,5"
\ar@{--}"1,6";"2,6"
\ar@{--}"1,3";"1,4"
\ar@{--}"2,3";"2,4"
}}
\qquad\qquad
\vcenter{
\xymatrix{
\la_1 & \mu_1\\
\la_a & \mu_a
\ar@/_/@{-}"1,1";"2,1"_{1_1}
\ar@/^/@{-}"1,1";"2,1"^{1_a}
\ar@/_/@{-}"1,2";"2,2"_{3_1}
\ar@/^/@{-}"1,2";"2,2"^{3_a}
\ar@{-}"1,1";"1,2"^{2_1}
\ar@{-}"2,1";"2,2"_{2_a}
}},
$$
and $\Ga(B_W)$ has no loops.
\end{exm}

\begin{exm}[Deletion of multiple edges]
Let $B$ be the following Brauer permutation with two pairs of double edges
in its Brauer graph $\Ga(B)$:
$$
\vcenter{
\xymatrix{
1^+ & 2^+ & 3^+ & 4^+\\
1^- & 2^- & 3^- & 4^-
\ar"1,1";"1,2"
\ar"1,2";"1,3"
\ar"1,3";"1,4"
\ar@/_15pt/"1,4";"1,1"_{}
\ar@/^/"2,1";"2,2"
\ar@/^/"2,2";"2,1"
\ar@/^/"2,3";"2,4"
\ar@/^/"2,4";"2,3"
\ar@{--}"1,1";"2,1"
\ar@{--}"1,2";"2,2"
\ar@{--}"1,3";"2,3"
\ar@{--}"1,4";"2,4"
}
}
\qquad\qquad
\xymatrix{
\la & \mu & \nu
\ar@{-}@/^/"1,1";"1,2"^1
\ar@{-}@/_/"1,1";"1,2"_2
\ar@{-}@/^/"1,2";"1,3"^4
\ar@{-}@/_/"1,2";"1,3"_3
}.
$$
Then the graph $\De$ defined in Proposition \ref{prp:del-double-one}
is a Dynkin graph $A_3$, which is a tree.
Therefore this proposition can be applied.
We choose $V'$ to be the set consisting of $\la:= \ang{\si}1^-$ and
$\nu:= \ang{\si}3^-$.
We then take $G:= \ang{a \mid a^2 = 1}$, and define an admissible
$G$-weight $W$ by
$$
\vcenter{
\xymatrix{
1^+ & 2^+ & 3^+ & 4^+\\
1^- & 2^- & 3^- & 4^-
\ar"1,1";"1,2"_1
\ar"1,2";"1,3"_1
\ar"1,3";"1,4"_1
\ar@/_15pt/"1,4";"1,1"_1
\ar@/^/"2,1";"2,2"^a
\ar@/^/"2,2";"2,1"^a
\ar@/^/"2,3";"2,4"^a
\ar@/^/"2,4";"2,3"^a
\ar@{--}"1,1";"2,1"
\ar@{--}"1,2";"2,2"
\ar@{--}"1,3";"2,3"
\ar@{--}"1,4";"2,4"
}}.
$$
Then $B_W$ and $\Ga(B_W)$ are as follows.
$$
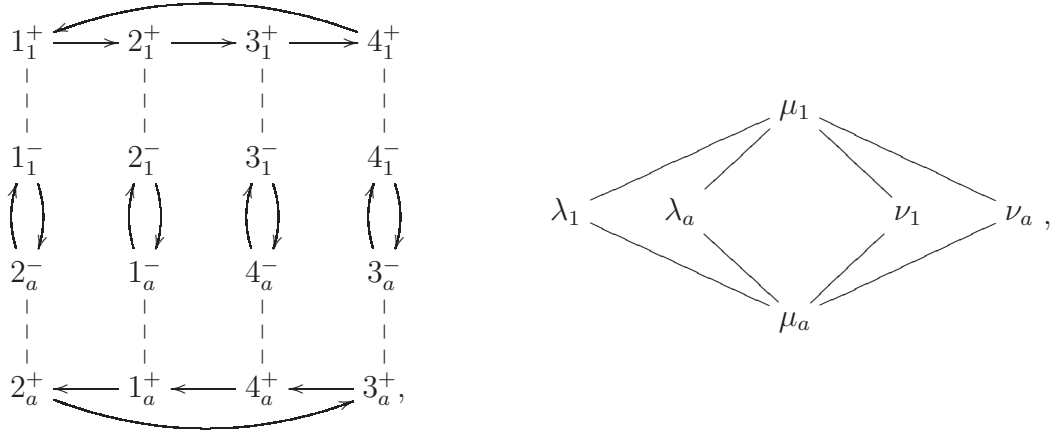
\vcenter{
\xymatrix{
1^+_1 & 2^+_1 & 3^+_1 & 4^+_1\\
1^-_1 & 2^-_1 & 3^-_1 & 4^-_1\\
2^-_a & 1^-_a & 4^-_a & 3^-_a\\
2^+_a & 1^+_a & 4^+_a & 3^+_a,
\ar"1,1";"1,2"
\ar"1,2";"1,3"
\ar"1,3";"1,4"
\ar"4,4";"4,3"
\ar"4,3";"4,2"
\ar"4,2";"4,1"
\ar@/_15pt/"1,4";"1,1"_{}
\ar@/_15pt/"4,1";"4,4"_{}
\ar@/^/"2,1";"3,1"
\ar@/^/"3,1";"2,1"
\ar@/^/"2,2";"3,2"
\ar@/^/"3,2";"2,2"
\ar@/^/"2,3";"3,3"
\ar@/^/"3,3";"2,3"
\ar@/^/"2,4";"3,4"
\ar@/^/"3,4";"2,4"
\ar@{--}"1,1";"2,1"
\ar@{--}"1,2";"2,2"
\ar@{--}"1,3";"2,3"
\ar@{--}"1,4";"2,4"
\ar@{--}"3,1";"4,1"
\ar@{--}"3,2";"4,2"
\ar@{--}"3,3";"4,3"
\ar@{--}"3,4";"4,4"
}}
\qquad\qquad
\vcenter{
\xymatrix{
&&\mu_1\\
\la_1& \la_a &&\nu_1 & \nu_a\\
&&\mu_a
\ar@{-}"1,3";"2,1"
\ar@{-}"1,3";"2,2"
\ar@{-}"1,3";"2,4"
\ar@{-}"1,3";"2,5"
\ar@{-}"3,3";"2,1"
\ar@{-}"3,3";"2,2"
\ar@{-}"3,3";"2,4"
\ar@{-}"3,3";"2,5"
}},
$$
and $\Ga(B_W)$ has no multiple edges
\end{exm}

\begin{exm}[Smash product with a non-abelian group]
Let $B$ be the Brauer permutation given in Example \ref{exm:first} and
$G:= \ang{a, b \mid a^3 = 1, b^2 = 1, aba = b}$ the symmetric group of
order 6.
We define an admissible $G$-weight $W$ by
$$
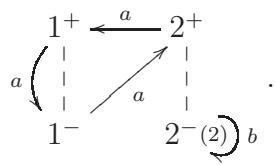
\vcenter{
\xymatrix{
1^+ & 2^+\\
1^- & 2^-\ 
\ar"1,2";"1,1"_a
\ar"2,1";"1,2"_a
\ar@/_1pc/"1,1";"2,1"_a
\ar@{--}"1,1";"2,1"
\ar@{--}"1,2";"2,2"
\ar@(ur,dr)_{(2)}^b
}}.
$$
Then $B_W$ is given as follows:
$$
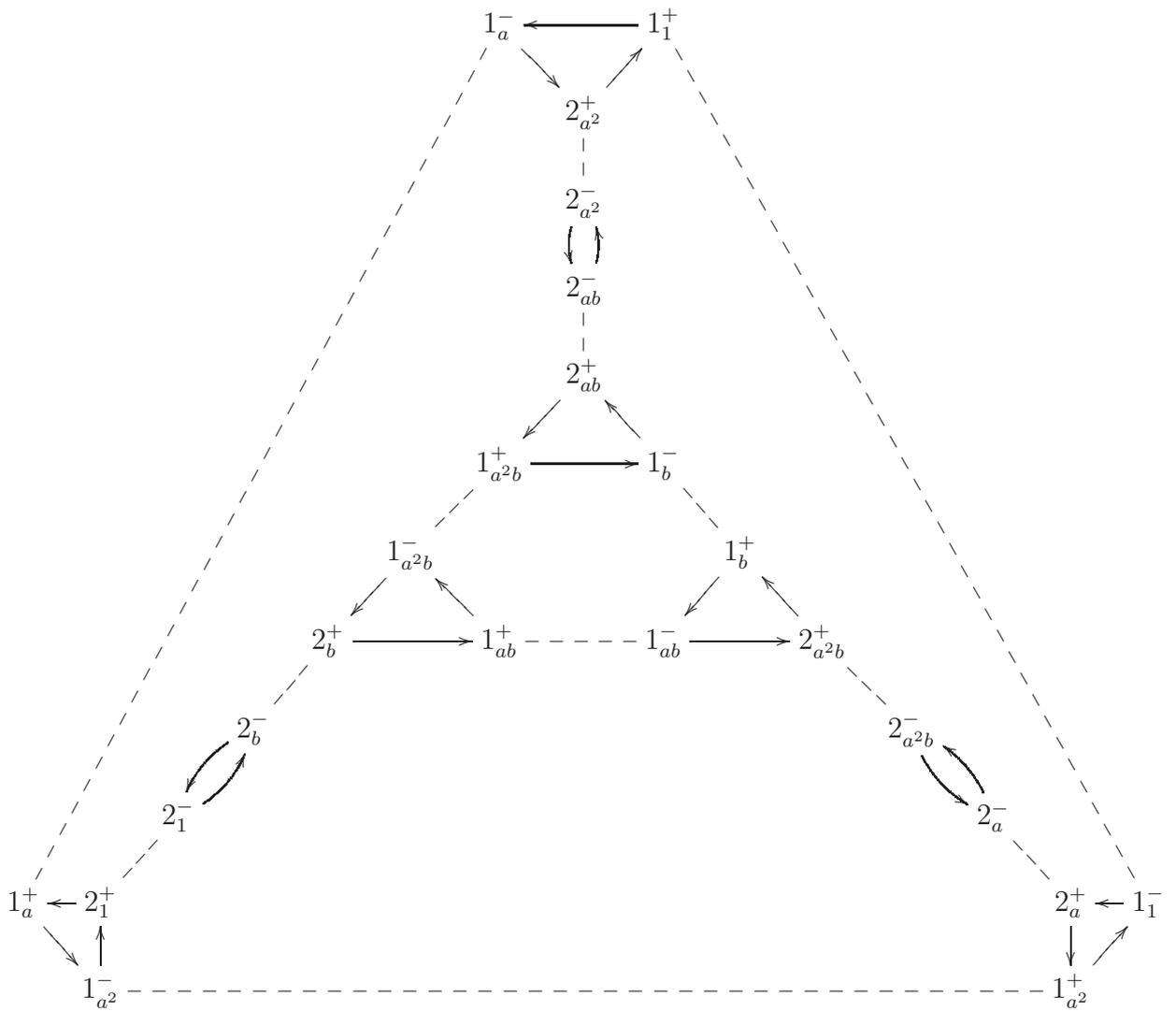
\xymatrix@R=15pt@C=10pt{
&&&&&&1^-_a && 1^+_1\\
&&&&&&&2^+_{a^2}\\
&&&&&&&2^-_{a^2}\\
&&&&&&&2^-_{ab}\\
&&&&&&&2^+_{ab}\\
&&&&&&1^+_{a^2b}&&1^-_b\\
&&&&&1^-_{a^2b}&&&&1^+_b\\
&&&&2^+_b&&1^+_{ab} &&1^-_{ab} &&2^+_{a^2b}\\
&&&2^-_b&&&&&&&&2^-_{a^2b}&\\
&&2^-_1&&&&&&&&&&2^-_a\\
1^+_a&2^+_1 &&&&&&&&&&&& 2^+_a & 1^-_1\\
&1^-_{a^2} &&&&&&&&&&&& 1^+_{a^2}
\ar"1,9";"1,7" \ar"1,7";"2,8" \ar"2,8";"1,9"
\ar@{--}"2,8";"3,8"
\ar@/_/"3,8";"4,8" \ar@/_/"4,8";"3,8"
\ar@{--}"4,8";"5,8"
\ar"5,8";"6,7" \ar"6,7";"6,9" \ar"6,9";"5,8"
\ar@{--}"6,7";"7,6" \ar@{--}"6,9";"7,10"
\ar"7,6";"8,5" \ar"8,5";"8,7" \ar"8,7";"7,6" \ar"7,10";"8,9" \ar"8,9";"8,11" \ar"8,11";"7,10" \ar@{--}"8,7";"8,9"
\ar@{--}"8,5";"9,4" \ar@{--}"8,11";"9,12"
\ar@/_/"9,4";"10,3" \ar@/_/"10,3";"9,4" \ar@/_/"9,12";"10,13" \ar@/_/"10,13";"9,12"
\ar@{--}"10,3";"11,2" \ar@{--}"10,13";"11,14"
\ar"11,2";"11,1" \ar"11,1";"12,2" \ar"12,2";"11,2"
\ar"11,14";"12,14" \ar"12,14";"11,15" \ar"11,15";"11,14"
\ar@{--}"1,7";"11,1" \ar@{--}"12,2";"12,14" \ar@{--}"11,15";"1,9"}
$$
Therefore $\Ga(B) = \xymatrix@C=40pt{
\la \ar@{-}@(dl,ul)^1 & \mu 
\ar@{-}"1,1";"1,2"^2^(1){(2)}
}$ changes to $\Ga(B_W)$ below:
$$
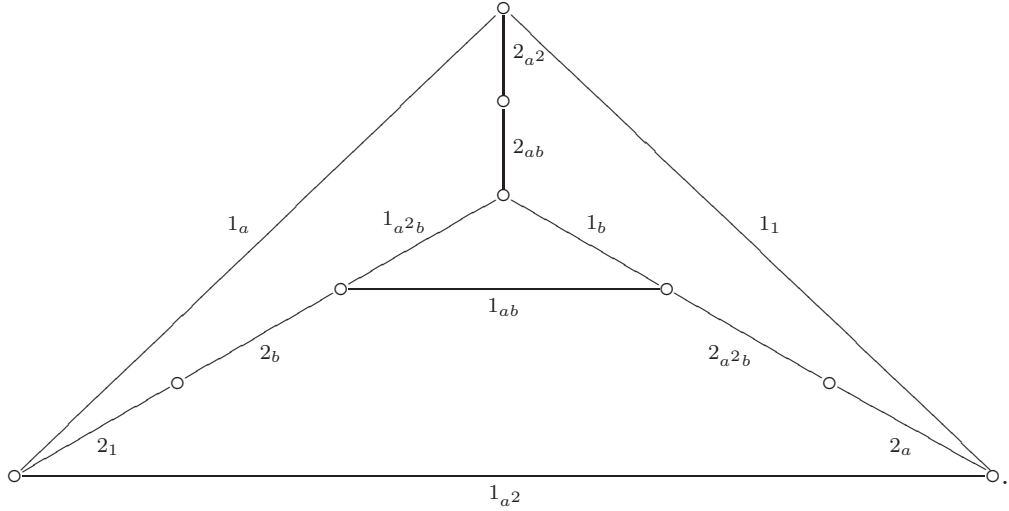
\qquad\qquad
\vcenter{
\xymatrix@R=30pt@C=55pt@M=0pt{
&&&\circ\\
&&& \circ\\
&&&\circ\\
&&\circ&& \circ\\
&\circ&&&& \circ\\
\circ&&&&&&\circ.
\ar@{-}"1,4";"2,4"^{2_{a^2}}
\ar@{-}"2,4";"3,4"^{2_{ab}}
\ar@{-}"3,4";"4,3"_{1_{a^2b}} \ar@{-}"3,4";"4,5"^{1_b} \ar@{-}"4,3";"4,5"_{1_{ab}}
\ar@{-}"4,3";"5,2"^{2_b} \ar@{-}"4,5";"5,6"_{2_{a^2b}}
\ar@{-}"5,2";"6,1"^{2_1} \ar@{-}"5,6";"6,7"_{2_a}
\ar@{-}"1,4";"6,1"_{1_a} \ar@{-}"6,1";"6,7"_{1_{a^2}} \ar@{-}"6,7";"1,4"_{1_1}
}}
$$
\end{exm}

\begin{exm}[Smash product with an infinite group: deletion of a cycle]
\label{exm:del-cycle}
Take $G:=\ang{a}$ to be the infinite cyclic group,
and let $(B, W)$ the following Brauer permutation with an admissible $G$-weight
($\Ga(B)$ is presented on the right)
$$
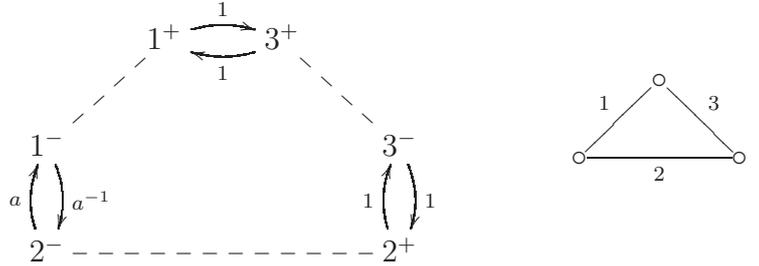
\vcenter{
\xymatrix{
& 1^+ & 3^+\\
1^-&&& 3^-\\
2^- &&&2^+
\ar@/^/"1,2";"1,3"^1 \ar@/^/"1,3";"1,2"^1
\ar@/^/"2,1";"3,1"^{a\inv} \ar@/^/"3,1";"2,1"^a
\ar@/^/"2,4";"3,4"^1 \ar@/^/"3,4";"2,4"^1
\ar@{--}"1,2";"2,1" \ar@{--}"3,1";"3,4" \ar@{--}"2,4";"1,3"
}
}\qquad\qquad
\vcenter{
\xymatrix@M=0pt{
&\circ\\
\circ && \circ
\ar@{-}"1,2";"2,1"_1 \ar@{-}"2,1";"2,3"_2 \ar@{-}"2,3";"1,2"_3
}}\ .
$$
Then $B_W$ and $\Ga(B_W)$ are given as follows, respectively:
$$
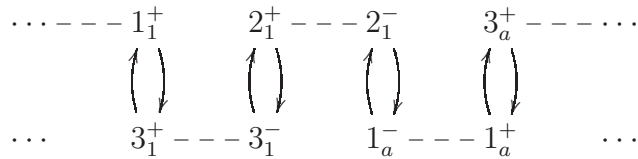
\xymatrix{
\cdots & 1^+_1 & 2^+_1 & 2^-_1 & 3^+_a & \cdots\\
\cdots & 3^+_1 & 3^-_1 & 1^-_a & 1^+_a & \cdots
\ar@{--}"1,1";"1,2" \ar@{--}"1,3";"1,4" \ar@{--}"1,5";"1,6"
\ar@{--}"2,2";"2,3" \ar@{--}"2,4";"2,5"
\ar@/^/"1,2";"2,2" \ar@/^/"2,2";"1,2"
\ar@/^/"1,3";"2,3" \ar@/^/"2,3";"1,3"
\ar@/^/"1,4";"2,4" \ar@/^/"2,4";"1,4"
\ar@/^/"1,5";"2,5" \ar@/^/"2,5";"1,5"
}
$$
and
$$
\xymatrix@M=0pt{
\cdots\  & \circ & \circ & \circ & \circ &\  \cdots
\ar@{-}"1,1";"1,2"^{1_1}
\ar@{-}"1,2";"1,3"^{3_1}
\ar@{-}"1,3";"1,4"^{2_1}
\ar@{-}"1,4";"1,5"^{1_a}
\ar@{-}"1,5";"1,6"^{3_a}
}.
$$
In this example the 3-cycle $\Ga(B)$ is transformed to an infinite Brauer tree $\Ga(B_W)$.
by the infinite cyclic group, which is smaller than
the construction in Proposition \ref{prp:del-cycle}.
\end{exm}

\bibliographystyle{plain}

\end{document}